\documentclass[11pt]{amsart}
\input epsf
\usepackage{latexsym}
\usepackage{amssymb,amsmath,amsthm,amscd, amssymb,
graphicx, enumerate, verbatim}
\usepackage[all]{xy}

\numberwithin{equation}{section}
\newtheorem{cor}[equation]{Corollary}
\newtheorem{lem}[equation]{Lemma}
\newtheorem{prop}[equation]{Proposition}
\newtheorem{thm}[equation]{Theorem}

\newtheorem{quest}[equation]{Question}
\newtheorem{fact}[equation]{Fact}

\newtheorem{Example}[equation]{Example}
\newenvironment{ex}{\begin{Example}\rm}{\end{Example}}
\newtheorem{remark}[equation]{Remark}
\newenvironment{rmk}{\begin{remark}\rm}{\end{remark}}
\def\co{\colon\thinspace}

\newcommand{\Cr}{{\mathfrak C}}

\newcommand{\e}{\varepsilon}

\def\a{\alpha}

\def\g{\gamma}
\def\de{\delta}

\def\d{\partial}
\def\k{\kappa}
\def\th{\theta}
\def\r{\rho}

\def\s{\sigma}

\def\S1{\bf S^1}

\begin{document}

\abovedisplayskip=6pt plus3pt minus3pt
\belowdisplayskip=6pt plus3pt minus3pt

\title[Rays and souls in von Mangoldt planes]
{\bf Rays and souls in von Mangoldt  planes}

\rm
%53C20 Global Riemannian geometry, including pinching
%53C22 Geodesics
%53C45 View Publications (1973-now) Global surface theory 
%(convex surfaces à la A. D. Aleksandrov) 

\author{Igor Belegradek \and Eric Choi \and Nobuhiro Innami}

\address{Igor Belegradek\\School of Mathematics\\ Georgia Institute of
Technology\\ Atlanta, GA 30332-0160\\ USA}\email{ib@math.gatech.edu}

\address{Eric Choi\\Department of Mathematics and Computer Science\\ 
Emory University\\400 Dowman Dr., W401\\Atlanta, GA 30322\\ USA}
\email{ericchoi314@gmail.com}

\address{Nobuhiro Innami\\Department of Mathematics\\
Faculty of Science, Niigata University\\
Niigata, 950-2181\\JAPAN}

\email{innami@math.sc.niigata-u.ac.jp}

%\thanks{}

\date{}
\begin{abstract} 
We study rays in von Mangoldt planes, which has
applications to the structure of open complete manifolds with
lower radial curvature bounds. We prove that the set of souls of 
any rotationally symmetric plane of nonnegative curvature is 
a closed ball, and if the plane is von Mangoldt we compute
the radius of the ball.
We show that each cone in $\mathbb R^3$ can be smoothed
to a von Mangoldt plane.
\end{abstract}
\maketitle
%\tableofcontents

\section{Introduction}

Let $M_m$ denote $\mathbb R^2$ equipped with a 
smooth, complete, rotationally symmetric Riemannian metric 
given in polar coordinates as $g_m:=dr^2+m^2(r)d\th^2$; let $o$ denote
the origin in $\mathbb R^2$.
We say that $M_m$ is a {\it von Mangoldt plane\,} if its 
sectional curvature $G_m:=-\frac{m^{\prime\prime}}{m}$ is a non-increasing
function of $r$. 

The Toponogov comparison theorem was extended in~\cite{IMS-topon-vm} 
to open complete manifolds with radial sectional curvature 
bounded below by the curvature of a von Mangoldt plane,  
leading to various applications in~\cite{ShiTan-mathZ-2002, KonOht-2007, KonTan-I}
and generalizations in~\cite{MasShi, KonTan-II, Mac}.

A point $q$ in a Riemannian manifold is called a {\it critical point
of infinity\,} if each unit tangent vector at $q$ 
makes angle $\le\frac{\pi}{2}$ with a ray that starts at $q$.
Let $\Cr_m$ denote the set of critical points of infinity of $M_m$; clearly
$\Cr_m$ is a closed, rotationally symmetric subset that
contains every pole of $M_m$, so that $o\in\Cr_m$.
One reason for studying $\Cr_m$ is the following consequence of 
the generalized Toponogov theorem of~\cite{IMS-topon-vm}. 

\begin{lem}
\label{lem-intro: crit}
Let $\hat M$ be a complete noncompact Riemannian manifold
with radial curvature bounded below by
the curvature of a von Mangoldt plane $M_m$, and let $\hat r$, $r$ 
denote the distance functions to the basepoints $\hat o$, $o$ of $\hat M$, $M_m$, 
respectively. 
If $\hat q$ is a critical point of $\hat r$, then 
$\hat r(\hat q)$ is contained in $r(\Cr_m)$.
\end{lem}

Combined with the critical point theory of distance 
functions~\cite{Grove-crit-pt-1993},
\cite[Lemma 3.1]{Gre-surv}, \cite[Section 11.1]{Pet-book}, Lemma~\ref{lem-intro: crit}
implies the following.

\begin{cor}\label{cor-intro: crit pt theory}  
In the setting of Lemma~\textup{\ref{lem-intro: crit}},
for any $c$ in $[a,b]\subset r(M_m\mbox{--}\,\Cr_m)$, 
\vspace{-10pt}
\newline
$\bullet$
the ${\hat r}^{-1}$-preimage of $[a,b]$ is homeomorphic to 
${\hat r}^{-1}(a)\times [a,b]$, and
the \newline \phantom{$\bullet$}
${\hat r}^{-1}$-preimages  of points in $[a,b]$
are all homeomorphic;\vspace{3pt} \newline
$\bullet$
the $\hat r^{-1}$-preimage of $[0,c]$ is homeomorphic to a compact
smooth manifold \newline\phantom{$\bullet$} with boundary, and the homeomorphism maps
$\hat r^{-1}(c)$ onto the boundary; \vspace{3pt}\newline
$\bullet$
if $K\subset\hat M$ is a compact 
smooth submanifold, possibly with boundary, such 
\, \phantom{$\bullet$} that
$\hat r(K)\supset r(\Cr_m)$, then $\hat M$ is diffeomorphic to
the normal bundle of $K$.  
\end{cor}

If $M_m$ is von Mangoldt and $G_m(0)\le 0$,
then $G_m\le 0$ everywhere, so every point is a pole and hence
$\Cr_m=M_m$ so that Lemma~\ref{lem-intro: crit} yields no information
about the critical points of $\hat r$. Of course, there are other ways
to get this information as illustrated by classical Gromov's
estimate: if $M_m$ is the standard $\mathbb R^2$,
then the set of critical points of $\hat r$ is compact; 
see e.g.~\cite[page 109]{Gre-surv}.

The following theorem determines $\Cr_m$ when $G_m\ge 0$ everywhere; 
note that the plane $M_m$ in (i)-(iii) need not be von Mangoldt.

\begin{thm} 
\label{thm-intro: nonneg curv}
If $G_m\ge 0$, then
\begin{enumerate}\vspace{-5pt} 
\item[\textup{(i)}] $\Cr_m$ is the closed $R_m$-ball centered at $o$ for some
$R_m\in [0,\infty]$.
\item[\textup{(ii)}] $R_m$ is positive if and only if
%\, $\displaystyle{\liminf_{r\to\infty}} m>0$ and
$\int_1^\infty m^{-2}$ is finite. 
\item[\textup{(iii)}] $R_m$ is finite if and only if $m^\prime(\infty)<\frac{1}{2}$.
\item[\textup{(iv)}]
If $M_m$ is von Mangoldt and $R_m$ is finite, 
then the equation $m^\prime(r)=\frac{1}{2}$ has
a unique solution $\r_m$, and the solution satisfies 
$\r_m> R_m$ and $G_m(r_m)>0$. 
\item[\textup{(v)}] 
If $M_m$ is von Mangoldt and
$R_m$ is finite and positive, then
$R_m$ is the unique solution of the integral equation 
$\int_x^\infty\frac{m(x)\, dr}{m(r)\,\sqrt{m^2(r)-m^2(x)}}=\pi$.
\end{enumerate}
\end{thm}

Here is a sample application of part (iv) of 
Theorem~\ref{thm-intro: nonneg curv}
and Corollary~\ref{cor-intro: crit pt theory}:

\begin{cor}
\label{cor-intro: homeo to .5-ball}
Let $\hat M$ be a complete noncompact Riemannian manifold
with radial curvature from the basepoint $\hat o$
bounded below by the curvature of a von Mangoldt plane $M_m$.
If $G_m\ge 0$ and $m^\prime(\infty)<\frac{1}{2}$, then $\hat M$ is homeomorphic
to the metric $\r_m$-ball centered at $\hat o$, where $\r_m$ is
the unique solution of $m^\prime(r)=\frac{1}{2}$.
\end{cor}

Theorem~\ref{thm-intro: nonneg curv}
should be compared with the following results of Tanaka:
\begin{itemize}
\item
the set of poles in any $M_m$ is a closed metric ball centered at $o$
of some radius $R_p$ in $[0,\infty]$~\cite[Lemma 1.1]{Tan-cut}.\vspace{3pt} 
\item
$R_p>0$ if and only if 
$\int_1^\infty m^{-2}$ is finite and 
$\underset{r\to\infty}{\liminf}\, m(r)>0$~\cite{Tan-ball-poles-ex}. \vspace{3pt}
\item
if $M_m$ is von Mangoldt, then $R_p$
is a unique solution of an explicit integral 
equation~\cite[Theorem 2.1]{Tan-ball-poles-ex}.
\end{itemize}

It is natural to wonder when 
the set of poles equals $\Cr_m$, and we answer the question 
when $M_m$ is von Mangoldt.

\begin{thm} 
\label{intro-thm: poles}
If $M_m$ is a von Mangoldt plane, then
\begin{enumerate}\vspace{-5pt} 
\item[\textup{(a)}]
If $R_p$ is finite and positive, then the set of poles
is a proper subset of the component of $\Cr_m$ 
that contains $o$. 
\item[\textup{(b)}]
$R_p=0$ if and only if $\Cr_m=\{o\}$.
\end{enumerate}
\end{thm}

Of course $R_p=\infty$ implies $\Cr_m=M_m$, but
the converse is not true: 
Theorem~\ref{thm-intro: slope realized} ensures
the existence of a von Mangoldt plane with 
$m^\prime(\infty)=\frac{1}{2}$ and $G_m\ge 0$, and
for this plane $\Cr_m=M_m$ by
Theorem~\ref{thm-intro: nonneg curv}, while $R_p$ is finite by
Remark~\ref{rmk: poles when m'=.5}.

We say that a ray $\g$ in $M_m$ {\it points away from infinity\,}
if $\g$ and the segment $[\g(0), o\,]$ make an angle $<\frac{\pi}{2}$ 
at $\g(0)$.
Define $A_m\subset M_m\,\mbox{--}\,\{o\}$ as follows: $q\in A_m$
if and only if there is a ray that starts at $q$ and points 
away from infinity; by symmetry, $A_m\subset\Cr_m$.

\begin{thm}
\label{thm-intro: away}
If $M_m$ is a von Mangoldt plane, then $A_m$ is open in $M_m$.
\end{thm}

Any plane $M_m$ with $G_m\ge 0$ has another distinguished subset, namely
the set of souls, i.e. points produced via the soul construction of
Cheeger-Gromoll. 

\begin{thm}
\label{thm-intro: crit is soul}
If $G_m\ge 0$, then
$\Cr_m$ is equal to the set of souls of $M_m$.
\end{thm}

Recall that the soul construction takes as input
a basepoint in an open complete manifold $N$ of nonnegative sectional 
curvature and produces a compact totally convex submanifold $S$
without boundary, called a {\it soul}, such that $N$ is diffeomorphic
to the normal bundle to $S$. Thus if $N$ is contractible,
as happens for $M_m$, then $S$ is a point.
The soul construction also gives a
continuous family of compact totally convex subsets
that starts with $S$ and ends with $N$, and 
according to~\cite[Proposition 3.7]{Men}
%Mendon{\c{c}}a, S{\'e}rgio
$q\in N$ is a critical point of infinity if and only if 
there is a soul construction such that
the associated continuous family of totally convex sets
drops in dimension at $q$.
In particular, any point of $S$ is a critical point of infinity,
which can also be seen directly; see the proof 
of~\cite[Lemma 1]{Mae-2rays}. In Theorem~\ref{thm-intro: crit is soul} 
we prove conversely that every point of $\Cr_m$ is a soul;
for this $M_m$ need not be von Mangoldt.

%If $q$ is a pole, then
%the soul construction that starts at $q$ produces the soul $S=\{q\}$.

In regard to part (iii) of Theorem~\ref{thm-intro: nonneg curv},
it is worth mentioning $G_m\ge 0$ implies that $m^\prime$ 
is non-increasing, so $m^\prime(\infty)$ exists,
and moreover, $m^\prime(\infty)\in [0,1]$ because $m\ge 0$.
As we note in Remark~\ref{rmk: app, tot curv}
for any von Mangoldt plane $M_m$, the limit 
$m^\prime(\infty)$ exists as a number in $[0,\infty]$.  
It follows that any $M_m$ with $G_m\ge 0$, and
any von Mangoldt plane $M_m$ admits
total curvature, which equals $2\pi(1-m^\prime(\infty))$
and hence takes values in $[-\infty, 2\pi]$; thus 
$m^\prime(\infty)=\frac{1}{2}$ if and only if $M_m$ 
has total curvature $\pi$.
Standard examples of von Mangoldt planes of positive curvature
are the one-parametric family of paraboloids, 
all satisfying $m^\prime(\infty)=0$~\cite[Example 2.1.4]{SST}, 
and the one-parametric family of two-sheeted hyperboloids
parametrized by $m^\prime(\infty)$,
which takes every value in $(0,1)$~\cite[Example 2.1.4]{SST}.

A property of von Mangoldt planes, discovered 
in~\cite{Ele, Tan-cut} and crucial to this paper, is
that the cut locus of any $q\in M_m\,\mbox{--}\, \{o\}$ is a ray 
that lies on the meridian opposite $q$. 
(If $M_m$ is not von Mangoldt, its cut locus is not fully
understood, but
it definitely can be disconnected~\cite[page 266]{Tan-ball-poles-ex},
and known examples of cut loci of compact
surfaces of revolution~\cite{GluSin, SinTan-experim} suggest that it
could be complicated). 

As we note in Lemma~\ref{lem: crit point and ray tangent to parallel},
if $M_m$ is a von Mangoldt plane, and if $q\neq o$, 
then $q\in\Cr_m$ if and only if the geodesic tangent 
to the parallel through $q$ is a ray.
Combined with Clairaut's relation 
this gives the following  
``choking'' obstruction for a point 
$q$ to belong to $\Cr_m$
(see Lemma~\ref{lem: basic choke}):

\begin{prop}
\label{prop-intro: choking} 
If $M_m$ is von Mangoldt and $q\in \Cr_m$, then 
$m^\prime(r_q)> 0$ and $m(r)>m(r_q)$ for $r>r_q$,
where $r_q$ is the $r$-coordinate of $q$.
\end{prop}

The above proposition is immediate from Lemmas~\ref{lem: basic choke}
and \ref{lem: crit point and ray tangent to parallel}.
We also show in Lemma~\ref{lem: m growth} that if
$M_m$ is von Mangoldt and $\Cr_m\neq o$, then
there is $\r$ such that $m(r)$ is increasing and 
unbounded on $[\r,\infty)$.

The following theorem collects most of what we know about $\Cr_m$ for
a von Mangoldt plane $M_m$ with some negative curvature, where
the case $\underset{r\to\infty}{\liminf}\, m(r)=0$
is excluded because then $\Cr_m=\{o\}$ by 
Proposition~\ref{prop-intro: choking}.

\begin{thm}
\label{thm-intro: neg curv}
If $M_m$ is a von Mangoldt plane with a point where $G_m<0$
and such that $\underset{r\to\infty}{\liminf}\, m(r)>0$, then 
\begin{enumerate}\vspace{-5pt} 
\item[\textup{(1)}]
$M_m$ contains a line and has total curvature $-\infty$;
\item[\textup{(2)}]
if $m^\prime$ has a zero, then neither $A_m$ nor $\Cr_m$ is connected;
\item[\textup{(3)}]
$M_m\,\mbox{--}\, A_m$ is a bounded subset of $M_m$;
\item[\textup{(4)}]
the ball of poles of $M_m$ has positive radius.
\end{enumerate}
\end{thm}

In Example~\ref{ex: m' vanishes} we construct a von Mangoldt plane
$M_m$ to which part (2) of Theorem~\ref{thm-intro: neg curv} applies.
In Example~\ref{ex: m' positive but non-connected} we produce
a von Mangoldt plane $M_m$ such that neither $A_m$ nor $\Cr_m$ is connected
while $m^\prime>0$ everywhere.
We do not know whether there is a von Mangoldt plane 
such that $\Cr_m$ has more than two connected components.

Because of Lemma~\ref{lem-intro: crit} and 
Corollary~\ref{cor-intro: crit pt theory}, one is
interested in subintervals of $(0, \infty)$ 
that are disjoint from $r(\Cr_m)$,
as e.g. happens for any interval on which $m^\prime\le 0$, 
or for the interval $(R_m, \infty)$ in Theorem~\ref{thm-intro: nonneg curv}.
To this end we prove the following result, which
is a consequence of Theorem~\ref{thm: neck}.

\begin{thm}\label{intro-thm: neck nneg}
Let $M_n$ be a von Mangoldt plane with $G_n\ge 0$,
$n(\infty)=\infty$, and such that $n^\prime(x)<\frac{1}{2}$ for some $x$. 
Then for any $z>x$ there exists $y>z$ such that 
if $M_m$ is a von Mangoldt plane with
$n=m$ on $[0, y]$, then $r(\Cr_m)$ and $[x,z]$ are disjoint.
\end{thm}

In general, if $M_m$, $M_n$ are von Mangoldt planes with
$n=m$ on $[0, y]$, then the sets $\Cr_m$, $\Cr_n$ could be quite
different. For instance, if $M_n$ is a paraboloid, then $\Cr_n=\{o\}$, 
but by Example~\ref{ex: m' positive but non-connected}
for any $y>0$ there is a von Mangoldt $M_m$ with some negative curvature
such that $m=n$ on $[0,y]$, and by Theorem~\ref{thm-intro: neg curv}
the set $M_m\,\mbox{--}\, \Cr_m$ is bounded 
and $\Cr_m$ contains the ball of poles of positive radius.

Basic properties of von Mangoldt planes are described in 
Appendix~\ref{sec: vm planes}, in particular,
in order to construct a von Mangoldt plane with prescribed $G_m$
it suffices to check that $0$ is the only zero of
the solution of the Jacobi initial value problem (\ref{form: ode})
with $K=G_m$, where $G_m$ is smooth on $[0,\infty)$.
Prescribing values of $m^\prime$ is harder.
It is straightforward to see that if $M_m$ is a von Mangoldt plane
such that $m^\prime$ is constant near infinity, then 
$G_m\ge 0$ everywhere and $m^\prime(\infty)\in [0,1]$. 
We do not know whether there is
a von Mangoldt plane with $m^\prime=0$ near infinity,
but all the other values in $(0,1]$ can be prescribed:

\begin{thm}
\label{thm-intro: slope realized}
For every $s\in (0,1]$ there is $\r>0$ and a von Mangoldt plane $M_m$
such that $m^\prime=s$ on $[\r,\infty)$.
\end{thm}

Thus each cone in $\mathbb R^3$ can be smoothed
to a von Mangoldt plane, but we do not know how to construct
a (smooth) capped cylinder that is von Mangoldt.

{\bf Structure of the paper.} 
We collect notations and conventions in Section~\ref{sec: notations}.
Properties of von Mangoldt planes are reviewed in
Appendix~\ref{sec: vm planes}, while 
Appendix~\ref{sec: calc lem} contains a calculus lemma
relevant to continuity and smoothness of the turn angle. 
Section~\ref{sec: turn angle} contains various results
on rays in von Mangoldt planes, including the proof of 
Theorem~\ref{thm-intro: away} and
Proposition~\ref{prop-intro: choking}. 
Planes of nonnegative curvature are discussed in 
Section~\ref{sec: nneg}, where
Theorems~\ref{thm-intro: nonneg curv} and \ref{thm-intro: crit is soul}
are proved. A proof of
Theorem~\ref{thm-intro: slope realized} 
is in Section~\ref{sec: smoothed cones}, and the other results
stated in the introduction are proved in Section~\ref{sec: other proofs}.

{\bf Acknowledgments.} 
Belegradek is grateful to NSF for support (DMS-0804038).
This paper will be a part of Choi's Ph.D. thesis at Emory University.
\begin{comment}
We greatly appreciate comments by the referee, 
who pointed out 
a number of inaccuracies, suggested 
Lemmas~\ref{lem: crit is o}, 
\ref{lem: approx by segments}, \ref{lem: max deviat}, and showed us 
how to strengthen parts (i)-(iii) of 
Theorem~\ref{thm-intro: nonneg curv},
which we originally proved only when $M_m$ was von Mangoldt.
\end{comment}

\section{Notations and conventions.}
\label{sec: notations}

All geodesics are parametrized by arclength. 
Minimizing geodesics are called {\it segments}.
Let $\d_r$, $\d_\th$ denote the vector fields dual
to $dr$, $d\th$ on $\mathbb R^2$.
Given $q\neq o$, denote its polar coordinates by $\th_q$, $r_q$.
Let $\g_q$, $\mu_q$, $\tau_q$ denote the geodesics defined
on $[0,\infty)$ that start at $q$
in the direction of $\d_\th$, $\d_r$, $-\d_r$, respectively. 
We refer to $\tau_{q}\vert_{(r_q,\infty)}$ as the 
{\it meridian opposite $q$}; note that $\tau_{q}(r_q)=o$.
Also set $\k_{\g(s)}:=\angle(\dot\g(s),\d_r)$.

We write $\dot{r}$, $\dot\th$, $\dot\g$, $\dot\k$ for the derivatives of
$r_{\g(s)}$, $\th_{\g(s)}$, $\g(s)$, $\k_{\g(s)}$ by $s$, and 
write $m^\prime$ for $\frac{dm}{dr}$; similar notations are used for higher derivatives.

Let $\hat\k(r_q)$ denote the maximum of the angles
formed by $\mu_q$ and rays emanating from $q\neq o$;
let $\xi_q$ denote the ray with $\xi_q(0)=q$ 
for which the maximum is attained, i.e. 
such that $\k_{\xi_q(0)}=\hat\k(r_{q})$.

A geodesic $\g$ in $M_m\,\mbox{--}\, \{o\}$ is called
{\it counterclockwise\,} if $\dot\th>0$ and
{\it clockwise\,} if $\dot\th<0$. 
A geodesic in $M_m$ is clockwise, counterclockwise,
or can be extended to a geodesic through $o$. 
If $\g$ is clockwise, then it can be mapped to a 
counterclockwise geodesic by an isometric involution of $M_m$.

{\bf Convention:} {\it unless stated otherwise, 
any geodesic in $M_m$ that we consider is either 
tangent to a meridian or counterclockwise}.

Due to this convention the Clairaut constant and the turn angle
defined below are nonnegative, which will simplify notations.

\section{Turn angle and rays in $M_m$}
\label{sec: turn angle}

This section collects what we know about rays in $M_m$
with emphasis on the cases when $G_m\ge 0$ or $G_m^\prime\le 0$.
Let $\g$ be a geodesic in $M_m$ that does not pass through $o$, so that 
$\g$ is a solution of the geodesic equations
\begin{equation}\label{form: geod eq}
{\ddot r}=m\, m^\prime\, {\dot\th}^2
\qquad\qquad
{\dot\th}\, m^2=c
\end{equation}
where $c$ is called {\it Clairaut's constant of $\g$}. 
The equation $\dot{\th}\, m^2=c$ is the so called
{\it Clairaut's relation\,}, which since $\g$ is assumed counterclockwise, 
can be written as $c=m(r_{\g(s)})\sin\k_{\g(s)}$.
Thus $0\le c\le m(r_\g(s))$ where $c= m(r_\g(s))$ only
at points where $\g$ is tangent to a parallel, and
$c=0$ when $\g$ is tangent to a meridian.

A geodesic is called {\it escaping} if its image is unbounded, e.g.
any ray is escaping.

\begin{fact}
\label{fact: on escaping geodesics}
\rm\hspace{-10pt}
\begin{enumerate}
\item
A parallel through $q$ is a geodesic in $M_m$ if and only if 
$m^\prime(r_q)=0$~\cite[Lemma 7.1.4]{SST}.
\item
A geodesic $\g$ in $M_m$ is tangent to a parallel at $\g(s_0)$
if and only if ${\dot r}_{\g(s_0)}=0$. 
\item
\label{fact: escaping geodesics are tangent to parallels only once}
If $\g$ is a geodesic in $M_m$ and
${\dot r}_{\g(s)}$ vanishes more than once,
then $\g$ is invariant under a rotation of $M_m$ about 
$o$~\cite[Lemma 7.1.6]{SST} and hence not escaping.
\end{enumerate}
\end{fact}

\begin{lem} 
\label{lem: basic choke}
If $\g_q$ is escaping, 
then $m(r)>m(r_q)$ for $r>r_q$, and 
$m^\prime(r_q)> 0$. 
\end{lem}
\begin{proof}
Since $\g_q$ is escaping, the image of $s\to r_{\g_q}(s)$ 
contains $[r_q,\infty)$, and $q$ is the only point where
$\g_q$ is tangent to a parallel.
The Clairaut constant of $\g_q$ is $c=m(r_q)$, hence $m(r)>m(r_q)$
for all $r>r_q$. It follows that $m^\prime(r_q)\ge 0$. Finally,
$m^\prime(r_q)\neq 0$ else $\g_q$ would equal the parallel through 
$q$.
\end{proof}

\begin{lem} 
\label{lem: no crossing}
If $\g$ is escaping geodesic that is tangent to the parallel $P_q$
through $q$, then $\g\setminus\{q\}$ lies
in the unbounded component of $M_m\setminus P_q$.
\end{lem}
\begin{proof} 
By reflectional symmetry and uniqueness of geodesics,
$\g$ locally stays on the same side of the parallel 
$P_q$ through $q$, i.e. $\g$ is the union of $\g_q$ and
its image under the reflecting fixing $\mu_q\cup\tau_q$.
If $\g $ could 
cross to the other side of $P_q$ at some
point $\g(s)$, then 
$|r_{\g (s)}-r_q|$ would attain a maximum between 
$\g(s)$ and $q$, and at the maximum point
$\g $ would be tangent to a parallel. Since $\g$ is escaping, 
it cannot be tangent to parallels more than once, hence 
$\g$ stays on the same side of $P_q$ at all times, and 
since $\g$ is escaping, it stays in the unbounded component of
$M_m\setminus P_q$.
\end{proof}

For a geodesic $\g\co (s_1, s_2)\to M_m$ that does 
not pass through $o$, 
we define the {\it turn angle $T_\g$ of $\g$} as 
\[T_\g:=\int_\g d\th=
\int_{s_1}^{s_2}{\dot\th}_{\g(s)} ds=\th_{\g(s_2)}-\th_{\g(s_1)}.\] 
The Clairaut's relation reads ${\dot\th}=c/m^2\ge 0$
so the above integral $T_\g$ converges to a number in $[0,\infty]$.
Since $\g$ is unit speed, we have $({\dot r})^2+m^2{\dot\th}^2=1$. 
Combining this with ${\dot\th}=c/m^2$ gives
${\dot r}=\mathrm{sign}({\dot r})\sqrt{1-\frac{c^2}{m^2}}$, which
yields a useful formula for the turn angle:
if $\g$ is not tangent to a meridian or a parallel on $(s_1, s_2)$,
so that $\mathrm{sign}({\dot r}_{\g(s)})$ is a nonzero constant, then
\begin{equation}
\label{form: defn F_c}
\frac{d\th}{dr}=
\frac{{\dot\th}}{{\dot r}}=
\mathrm{sign}({\dot r}_{\g(s)})\ \! F_c(r)\ \ \ \text{where}\ 
\ \ F_c:=\frac{c}
{m\sqrt{m^2-c^2}},
\end{equation}
and thus if $r_i:=r_{\g(s_i)}$, then 
\begin{equation}
\label{form: general turn angle}
T_\g=
\mathrm{sign}({\dot r})\int_{r_1}^{r_2} F_c(r) dr.
\end{equation}
Since $c^2\le m^2$, this integral is finite except possibly when
some $r_i$ is in the set $\{m^{-1}(c),\, \infty\}$.
The integral (\ref{form: general turn angle}) 
converges at $r_i=m^{-1}(c)$ if and only if
$m^\prime(r_i)\neq 0$. 
Convergence of (\ref{form: general turn angle}) 
at $r_i=\infty$ implies convergence of $\int_1^\infty m^{-2} dr$, 
and the converse holds under the assumption
$\underset{r\to\infty}{\liminf}\, m(r)>c$;
this assumption is true when $G_m\ge 0$ or $G_m^\prime\le 0$,
as follows from Lemma~\ref{lem: m growth} below.

\begin{ex}
\label{ex: ray turn angle} 
If $\g$ is a ray in $M_m$ that does not pass through $o$,
then $T_\g\le \pi$ else there is $s$ with 
$|\th_{\g(s)}-\th_{\g(0)}|=\pi$, and by symmetry 
the points $\g(s)$, $\g(0)$ are joined by two segments, so $\g$
would not be a ray. 
\end{ex}

\begin{ex}
\label{ex: ray exists implies properties of m}
If $T_{\g_q}$ is finite, then $m^\prime(r_q)\neq 0$ and
$m^{-2}$ is integrable on $[1,\infty)$, as follows immediately from
the discussion preceding Example~\ref{ex: ray turn angle}.
\end{ex}

\begin{lem} 
\label{lem: turn angle escaping}
If $\g\co [0,\infty)\to M_m$ is a geodesic
with finite turn angle, then $\g$ is escaping.
\end{lem}
\begin{proof} Note that $\g$ is tangent to
parallels in at most two points, for otherwise
$\g$ is invariant under a rotation about $o$, and hence
its turn angle is infinite. Thus after cutting off a portion
of $\g$ we may assume it is never tangent to a parallel, 
so that $r_{\g(s)}$ is monotone. 
By assumption $\th_{\g(s)}$ is bounded and increasing.
By Clairaut's relation
$m(r_{\g(s)})$ is bounded below, so that
$m(0)=0$ implies that $r_{\g(s)}$ is bounded below.
If $\g$ were not escaping, then $r_{\g(s)}$ would also be 
bounded above, so there would exist a limit of $(r_{\g(s)},\th_{\g(s)})$
and hence the limit of $\g(s)$ as $s\to\infty$, 
contradicting the fact that $\g$ has infinite length.
\end{proof}

\begin{lem} 
\label{lem: m growth} 
If $m^{-2}$ is integrable on $[1,\infty)$, then
\begin{enumerate}\vspace{-5pt} 
\item[\textup{(1)}]
the function $(r\log r)^{^{-\frac{1}{2}}}m(r)$ is unbounded;
\item[\textup{(2)}]
if $G_m\ge 0$, then $m^\prime>0$ for all $r$;
\item[\textup{(3)}]
if $M_m$ is von Mangoldt, then $m^\prime>0$ for all large $r$;
\item[\textup{(4)}]
if either $G_m\ge 0$ or $G_m^\prime\le 0$, then $m(\infty)=\infty$.
\end{enumerate}
\end{lem}
\begin{proof} 
%Since $\g_q$ has finite turn angle and $F_c\ge cm^{-2}$,
%we conclude that $m^{-2}$ is integrable on $[1,\infty)$. 
Since $m^{-2}$ is integrable,
the function $(r\log r)^{^{-\frac{1}{2}}}m(r)$ is unbounded,
and in particular, $m$ is unbounded. 
If $G_m\ge 0$ everywhere, then $m^\prime$ is non-increasing
with $m^\prime(0)=1$, and the fact that $m$ is unbounded
implies that $m^\prime>0$ for all $r$.
If $M_m$ is von Mangoldt, and $G_m(\r_0)< 0$, then $G_m< 0$ for $r\ge \r_0$, 
i.e. $m^\prime$ is non-decreasing on $[\r_0, \infty)$.
Since $m$ is unbounded, there is $\r>\r_0$ with $m(\r)>m(\r_0)$
so that $\int_{\r_0}^{\r} m^\prime=m(\r)-m(\r_0)>0$.
Hence $m^\prime$ is positive somewhere
on $(\r_0, \r)$, and therefore on $[\r, \infty)$.
Finally, since $m$ is an unbounded increasing function for
large $r$, the limit 
$\displaystyle{\lim_{r\to\infty}\, m(r)}=m(\infty)$ exists and
equals $\infty$.
\end{proof}

\begin{lem}
If $\g_q$ is escaping, then 
$\displaystyle{\liminf_{r\to\infty}\, m(r)}>m(r_q)$ 
if and only if there is a neighborhood $U$ of $q$
such that $\g_u$ is escaping for each $u\in U$.
\end{lem}
\begin{proof} 
First, recall that  $m(r)>m(r_q)$ for $r>r_q$
and $m^\prime(r_q)>0$ by Lemma~\ref{lem: basic choke}.
We shall prove the contrapositive:
$\displaystyle{\liminf_{r\to\infty}\, m(r)}=m(r_q)$
if and only if there is a sequence $u_i\to q$ such that
$\g_{u_i}$ is not escaping.

If there is a sequence $z_i\in M_m$ with
$r_{z_i}\to\infty$ and $m(r_{z_i})\to m(r_q)$,
then there are points $u_i\to q$ on $\mu_q$ with 
$m(r_{u_i})=m(r_{z_i})$.
If $\g_{u_i}$ is escaping, then it meets the parallel 
through $z_i$, so 
Clairaut's relation implies that $\g_{u_i}$ is tangent
to the parallels through $u_i$ and $z_i$, 
which cannot happen for an escaping geodesic. 

Conversely,
suppose there are $u_i\to q$ such that $\g_i:=\g_{u_i}$ is not escaping. 
Let $R_i$ be the radius of the smallest ball about $o$
that contains $\g_i$, and let $P_i$ be its boundary 
parallel. Note that $R_i\to \infty$ 
as $\g_i$ converges to $\g_q$ on compact sets
and $\g_q$ is escaping, and hence
$\displaystyle{\liminf_{r\to\infty}\, m(r)}=
\displaystyle{\lim_{r\to\infty}\, m(R_i)}$.
For each $i$
there is a sequence $s_{i,j}$ such that the $r$-coordinates of 
$\g_i(s_{i,j})$ converge to $R_i$, which implies 
$\k_{\g_i(s_{i,j})}\to \frac{\pi}{2}$
as $j\to\infty$ and $i$ is fixed. 
(Note that if $\g_i$ is tangent to $P_i$, then
$s_{i,j}$ is independent of $j$, namely, 
$\g(s_{i,j})$ is the point of tangency). 
By Clairaut's relation,
$m(R_i)=m(r_{u_i})$, hence 
$\displaystyle{\liminf_{r\to\infty}\, m(r)}=m(r_q)$.
\end{proof}

\begin{lem} \label{lem: ray iff at most pi}
If $M_m$ is von Mangoldt, then
a geodesic $\g\co [0,\infty)\to M_m\setminus\{o\}$ is 
a ray if and only if $T_\g\le \pi$.
\end{lem}
\begin{proof} The ``only if'' direction holds 
even when $M_m$ is not von Mangoldt by 
Example~\ref{ex: ray turn angle}. 
Conversely, if $\g$ is not a ray, then $\g$ meets the cut locus 
of $q$, which by~\cite{Tan-cut} is a subset of the 
opposite meridian $\tau_{\g(0)}\vert_{(r_{\g(0)},\infty)}$.
Thus $T_\g>\pi$. 
\end{proof}

\begin{lem}\label{lem: above rays are rays}
If $\g$ is a ray in a von Mangoldt plane,
and if $\s$ is a geodesic with $\s(0)=\g(0)$ and
$\k_{\g(0)}>\k_{\s(0)}$, then $\s$ is a ray and $T_\s\le T_\g$.
\end{lem}
\begin{proof} Set $q=\g(0)$.
If $\k_{\g(0)}=\pi$, then $\g=\tau_q$, so $\tau_q$ is a ray,
which in a von Mangoldt plane implies that $q$ is a 
pole~\cite[Lemma 7.3.1]{SST}, so that $\s$ is also a ray.
If $\k_{\g(0)}<\pi$  and $\s$ is not a ray, then $\s$ is minimizing
until it crosses the opposite meridian 
$\tau_q\vert_{(r_q,\infty)}$~\cite{Tan-cut}.
Near $q$ the geodesic $\s$ lies in the
region of $M_m$ bounded by $\g$ and $\mu_{q}$
hence before crossing the opposite meridian 
$\s$ must intersect $\g$ or $\mu_{q}$, so they would not be rays. 
Finally, $T_\s\le T_\g$ holds as $\s$ lies in the sector between
$\g$ and $\mu_{q}$.
\end{proof}

\begin{lem}\label{lem: crit point and ray tangent to parallel}
If $M_m$ is von Mangoldt and $q\neq o$, then $\g_q$ is a ray
if and only if  $q\in\Cr_m$.
\end{lem}
\begin{proof} 
If $\g_q$ is a ray, then $q\in\Cr_m$ by symmetry.
If $q\in\Cr_m$, then either $q$ is a pole and 
there is a ray in any direction, or
$q$ is not a pole. In the latter case
$\tau_q$ is not a ray~\cite[Lemma 7.3.1]{SST}, hence by the
definition of $\Cr_m$ there is a ray $\g$
with $\k_{\g(0)}\ge \frac{\pi}{2}$, so $\g_q$ is a ray 
by Lemma~\ref{lem: above rays are rays}.
\end{proof}

Recall that $\hat\k(r_q)$ is the maximum of the angles
formed by $\mu_q$ and rays emanating from $q\neq o$,
and $\xi_q$ is the ray for which the maximum is attained.
It is immediate from definitions that
$q\in\Cr_m$ if and only if
$\hat\k(r_q)\ge \frac{\pi}{2}$.
Lemmas~\ref{lem: crit is o}, 
\ref{lem: approx by segments}, \ref{lem: max deviat} below
were suggested by the referee.

\begin{lem} 
\label{lem: crit is o}
$\Cr_m\neq\{o\}$ if and only if\, 
$\displaystyle{\liminf_{r\to\infty}\, m}>0$ 
and $\int_1^\infty m^{-2}$ is finite. 
\end{lem}
\begin{proof} 
The ``if'' direction holds because by the main result 
of~\cite{Tan-ball-poles-ex} 
the assumptions imply that the ball of poles has a positive radius.
Conversely, if $q\in\Cr_m\,\mbox{--}\,\{o\}$,
then $\xi_q$ is a ray different from $\mu_q$.
By~\cite[Lemma 1.3, Proposition 1.7]{Tan-ball-poles-ex} 
if either $\displaystyle{\liminf_{r\to\infty}\, m}=0$ or
$\int_1^\infty m^{-2}=\infty$, then $\mu_q$
is the only ray emanating from $q$.  
\end{proof}

\begin{lem}\label{lem: approx by segments}
$\xi_q$ is the limit
of the segments $[q,\tau_q(s)]$ as $s\to\infty$.
\end{lem}
\begin{proof}
The segments $[q,\tau_q(s)]$ subconverge to a ray $\s$ that starts
at $q$. Since $\xi_q$ is a ray, it
cannot cross the opposite meridian $\tau_q\vert_{(r_q,\infty)}$.
As $[q,\tau_q(s)]$ and $\xi_q$ are minimizing, 
they only intersect at $q$, 
and hence the angle formed by $\mu_q$ and $[q,\tau_q(s)]$
is $\ge\hat\k(r_q)$. It follows that $\k_{\s(0)}\ge\hat\k(r_{q})$,
which must be an equality as $\hat\k(r_q)$ is a maximum, 
so $\s=\xi_q$.
\end{proof}

\begin{lem} 
\label{lem: max deviat} The function
$r\to\hat\k(r)$ is left continuous and
upper semicontinuous. In particular, 
the set $\{q: \hat\k(r_q)<\a\}$ is open for every $\a$. 
\end{lem} 
\begin{proof}
If $\hat\k$ is not left continuous at $r_q$, then  
there exists $\e>0$ and a sequence of points $q_i$ on $\mu_q$
such that $r_{q_i}\to r_q-$ and either
$\hat\k(r_{q_i})-\hat\k(r_q)>\e$ or $\hat\k(r_q)-\hat\k(r_{q_i})>\e$.
In the former case $\xi_{q_i}$ subconverge to a ray that makes larger angle with 
$\mu_q$ that $\xi_q$, contradicting maximality of $\hat\k(r_q)$.
In the latter case, $\xi_{q_i}$ intersects $\xi_q$ for some $i$.
Therefore, by Lemma~\ref{lem: approx by segments}
the segment $[q_i, \tau_q(s)]$ intersects 
$[q, \tau_q(s)]$ for large enough $s$ at 
a point $z\neq\tau_q(s)$, so $\tau_q(s)$ is a cut point of $z$
which cannot happen for a segment. This proves
that $\hat\k$ is left continuous.
A similar argument shows that 
$\displaystyle{\limsup_{r_{q_i}\to r_q+}\hat\k(r_{q_i})}\le \hat\k(r_q)$,
so that $\hat\k$ is upper semicontinuous, which implies that
$\{q: \hat\k(r_q)<\a\}$ is open for every $\a$.
\end{proof}

Lemmas~\ref{lem: ray iff at most pi},
\ref{lem: crit point and ray tangent to parallel}
imply that on a von Mangoldt plane
$\hat\k(r_q)\ge\frac{\pi}{2}$ 
if and only if $T_{\g_q}\le\pi$; the equivalence is sharpened 
in Theorem~\ref{thm: max devia vs turn angle}, whose proof occupies
the rest of this section.

\begin{lem}
\label{lem: rays and parallels}
If $\s$ is escaping and $0<\k_{\s(0)}\le \frac{\pi}{2}$, then 
$T_\s=\int_{r_{q}}^\infty F_{c}(r) dr$; moreover,
if $\k_{\s(0)}=\frac{\pi}{2}$, then $c=m(r_{q})$.
\end{lem}
\begin{proof} 
This formula for $T_\s$ is immediate from 
(\ref{form: general turn angle}) once it is shown that
$\s\vert_{(0,\infty)}$ is not tangent to a meridian 
or a parallel. If $\s\vert_{(0,\infty)}$ were tangent to a meridian,
$\k_{\s(0)}$ would be $0$ or $\pi$, which is not the case.
Since $\s$ is escaping,  
Fact~\ref{fact: on escaping geodesics} implies that
$\s$ is tangent to parallels at most once.
If $\k_{\s(0)}=\frac{\pi}{2}$, then $\s$ is tangent to
the parallel through $\s(0)$, and so $\s\vert_{(0,\infty)}$
is not tangent to a parallel.
Finally, if $\k_{\s(0)}< \frac{\pi}{2}$, then
$\s$ is not tangent to a parallel,
else it would be tangent to a parallel through $u$ with 
$r_u>r_q$, which would imply $r_{\s(s)}\le r_u$ for all $s$ 
by Lemma~\ref{lem: no crossing}, which cannot happen
for an escaping geodesic.
\end{proof}

To better understand
the relationship between $\hat\k(r_q)$ and $T_{\g_q}$,
we study how $T_\s$ depends on $\s$, or equivalently
on $\s(0)$ and $\k_{\s(0)}$, when $\s$ varies
in a neighborhood of a ray $\g_q$.

\begin{lem} 
\label{lem: turn angle of gamma_q is continuous}
If $G_m\ge 0$ or $G_m^\prime\le 0$,
then the function $u\to T_{\g_u}$ is continuous at each
point $u$ where $T_{\g_u}$ is finite.
\end{lem}
\begin{proof}
If $T_{\g_u}$ is finite, then $\g_u$ is escaping by 
Lemma~\ref{lem: turn angle escaping}, and hence
$T_{\g_u}=\int_{r_{u}}^\infty F_{m(r_u)}$
by Lemma~\ref{lem: rays and parallels}.
We need to show that this integral depends continuously 
on $r_u$. 

By Lemma~\ref{lem: basic choke},
Lemma~\ref{lem: m growth}, and
the discussion preceding Example~\ref{ex: ray turn angle},
the assumptions on $G_m$ and finiteness of $T_{\g_u}$
imply that $m(r)>m(r_u)$ for $r>r_u$, $m^{-2}$ is integrable, 
$m^\prime(r_u)> 0$, and $m(\infty)=\infty$. 
Hence there exists $\delta>r_u$ with
$m^\prime\vert_{[r_u,\delta]}>0$, and 
$m(r)>m(\delta)$ for $r>\delta$; it is clear
that small changes in $u$ do not affect $\delta$. 

Write $\int_{r_{u}}^\infty F_{m(r_u)}=
\int_{r_{u}}^\delta F_{m(r_u)}+\int_{\delta}^\infty F_{m(r_u)}$.
On $[r_u,\delta]$ we can write
$F_{m(r_u)}=h(r,r_u)(r-r_u)^{-\frac{1}{2}}$ for some 
smooth function $h$. Since 
$(r-r_u)^{-\frac{1}{2}}$ is the derivative
of $2(r-r_u)^{\frac{1}{2}}$,
one can integrate $F_{m(r_u)}$ by parts which easily
implies continuous dependence of
$\int_{r_{u}}^\delta F_{m(r_u)}$ on $r_u$.

Continuous dependence of $\int_{\delta}^\infty F_{m(r_u)}$ on
$r_u$ follows because $F_{m(r_u)}$ is continuous in $r_u$,
and is dominated by $Km^{-2}$ where $K$ is a positive constant
independent of small changes of $r_u$. 
\end{proof}

Next we focus on the case when 
$\s(0)$ is fixed, while $\k_{\s(0)}$ varies
near $\frac{\pi}{2}$. To get an explicit formula for $T_\s$ 
we need the following.

\begin{lem}
\label{lem: geod just below ray} 
If $M_m$ is von Mangoldt, and $\g_q$ is a ray, then
there is $\e>0$ such that
every geodesic $\s\co [0,\infty)\to M_m$  
with $\s(0)=q$ and 
$\k_{\s(0)}\in [\frac{\pi}{2}, \frac{\pi}{2}+\e]$ 
is tangent to a parallel exactly once, and if $u$ is the 
point where $\s$ is tangent to a parallel, then
$m^\prime>0$ on $[r_u, r_q]$.
\end{lem}
\begin{proof} 
If $\k_{\s(0)}=\frac{\pi}{2}$, then $\s=\g_q$, so
it is tangent to a parallel only at $q$,
as rays are escaping. 
If $\k_{\s(0)}>\frac{\pi}{2}$, then 
$\s$ converges to $\g_q$ on compact subsets as $\e\to 0$,  
so for a sufficiently small $\e$ the geodesic $\s$ 
crosses the parallel through $q$ at some point $\s(s)$
such that $\k_{\s(s)}<\frac{\pi}{2}$. Since $\g_q$ is a ray,
rotational symmetry and 
Lemma~\ref{lem: above rays are rays} imply that
$\s\vert_{[s,\infty)}$ is a ray, so $\s$ is escaping.
Thus $\s$ is tangent to a parallel at a point $u$
where $r_{\s(s)}$ attains a minimum, and is not 
tangent to a parallel at any other point by 
Fact~\ref{fact: on escaping geodesics}.
Finally, $r_u=\lim_{\e\to 0}r_q$, and since $m^\prime(r_q)>0$
by Proposition~\ref{prop-intro: choking},
we get $m^\prime>0$ on $[r_u, r_q]$ for small $\e$. 
\end{proof}

Under the assumptions of Lemma~\ref{lem: geod just below ray}
the Clairaut constant $c$ of $\s$ equals 
$m(r_u)=m(r_q)\sin \k_{\s(0)}$, 
and the turn angle of $\s$ is given by
\begin{equation}
\label{form: upward turn angle via c}
T_\s=\int_{r_q}^\infty F_{m(r_q)}(r) dr
\quad\text{if}\quad \k_{\s(0)}=\frac{\pi}{2}\quad\text{and}
\end{equation}
\begin{equation}
\label{form: turn angle via c}
T_\s=\int_{r_u}^\infty F_{c}(r) dr -
\int_{r_q}^{r_u} F_{c}(r) dr =
\int_{r_q}^\infty F_{c}(r) dr +
2\int_{r_u}^{r_q} F_{c}(r) dr 
\end{equation}
if $\frac{\pi}{2}<\k_{\s(0)}<\frac{\pi}{2}+\e$. 
These integrals converge, i.e. $T_\s$ is finite, as follows from 
Example~\ref{ex: ray exists implies properties of m},
and Lemmas~\ref{lem: m growth}, \ref{lem: geod just below ray}.

Since any geodesic $\s$ with $\s(0)=q$ and 
$\k_{\s(0)}\in [0, \frac{\pi}{2}+\e]$ 
has finite turn angle,
one can think of $T_\s$ as a function of $\k_{\s(0)}$
where $\s$ varies over geodesics with $\s(0)=q$
and $\k_{\s(0)}\in [\,0, \frac{\pi}{2}+\e]$.

\begin{lem} 
\label{lem: turn angle cont and diff}
If $M_m$ is von Mangoldt, and $\g_q$ is a ray, then
there is $\delta>\frac{\pi}{2}$ such that the function
$\k_{\s(0)}\to T_\s$ is continuous 
and strictly increasing on $[\frac{\pi}{2}, \delta]$, and
continuously differentiable on $(\frac{\pi}{2}, \delta]$;
moreover, the derivative of $T_\s$ is infinite at $\frac{\pi}{2}$.
\end{lem}
\begin{proof} The Clairaut constant $c$ of $\s$ equals 
$m(r_u)=m(r_q)\sin \k_{\s(0)}$, so
the assertion is immediate from (elementary but nontrivial)
Lemma~\ref{lem: calc} about continuity and differentiability
of the integrals 
(\ref{form: upward turn angle via c})-(\ref{form: turn angle via c}).
\end{proof}

\begin{thm}\label{thm: max devia vs turn angle} 
If $M_m$ is von Mangoldt and $q\neq o$, then 
\begin{enumerate}\vspace{-5pt} 
\item[\textup{(1)}]  
$\hat\k(r_q)>\frac{\pi}{2}$ if and only if $T_{\g_q}<\pi$.
\vspace{2pt}
\item[\textup{(2)}]   
$\hat\k(r_q)=\frac{\pi}{2}$ if and only if $T_{\g_q}=\pi$.
\vspace{-5pt}
\end{enumerate}
\end{thm}
\begin{proof} (1)
If $\hat\k(r_q)>\frac{\pi}{2}$, 
then any geodesic $\s$ with $\s(0)=q$ and 
$\k_{\s(0)}\in [\frac{\pi}{2}, \hat\k(r_q)]$
is a ray, and so has turn angle $\le \pi$.
By Lemma~\ref{lem: turn angle cont and diff} 
the turn angle is increasing at $\frac{\pi}{2}$, so $T_{\g_q}<\pi$.
Conversely, if $T_{\g_q}<\pi$, then by
Lemma~\ref{lem: turn angle cont and diff} 
the turn angle is continuous at $\frac{\pi}{2}$, so
any geodesic $\s$ with
$\s(0)=q$ and $\k_{\s(0)}$ near $\frac{\pi}{2}$ has
turn angle $<\pi$, and is therefore a ray, so 
$\hat\k(r_q)>\frac{\pi}{2}$.

(2) follows from (1) and the fact 
that $\hat\k(r_q)\ge\frac{\pi}{2}$ if and only if $T_{\g_q}\le\pi$.
\end{proof}

\begin{proof}[Proof of Theorem~\ref{thm-intro: away}]
By Theorem~\ref{thm: max devia vs turn angle} we know that
$q\in A_m$ if and only if $T_{\g_q}<\pi$, 
and by Lemma~\ref{lem: turn angle of gamma_q is continuous} the map
$u\to T_{\g_u}$ is continuous at $q$, so the set 
$\{u\in M_m\,|\, T_{\g_u}<\pi\}$
is open, and hence so is $A_m$.
\end{proof}

\begin{proof}[Another proof of Theorem~\ref{thm-intro: away}]
Fix $q\in A_m$ so that $T_{\g_q}<\pi$ by 
Theorem~\ref{thm: max devia vs turn angle}.
Fix $\e>0$ such that $\e+T_{\g_q}<\pi$. 
Let $P_q$ be the parallel through $q$.
Then there is a ray $\g$ with $\g(0)=q$ and $\k_{\g(0)}>\frac{\pi}{2}$
such that $\g$ intersects $P_q$ at points
$q$, $\g(t)$, and the turn angle of $\g\vert_{(0,t)}$ 
is $<\e$.

For an arbitrary sequence $q_i\to q$ we need to show that
$q_i\in A_m$ for all large $i$.
Let $\g_i\co [0,\infty)\to M_m$ be the geodesic with 
$\g_i(0)=q_i$ and $\k_{\g_i(0)}=\k_{\g(0)}$.
Since $\g_i$ converge to $\g$ on compact sets, for large $i$ 
there are $t_i>0$ such that $\g_i(t_i)\in P_q$ and $t_i\to t$.
The angle formed by $\g$ and $\mu_{\g(t)}$
is $<\frac{\pi}{2}$. Rotational symmetry and 
Lemma~\ref{lem: above rays are rays} imply that 
if $i$ is large, then $\g_i\vert_{[t_i,\infty)}$
is a ray whose turn angle is $\le T_{\g_q}$. 
The turn angles of $\g_i\vert_{(0,t_i)}$ converge to the turn angle
of $\g\vert_{(0,t)}$, which is $<\e$. Thus $T_{\g_i}<T_{\g_q}+\e<\pi$
for large $i$, so that $\g_i$ is a ray by 
Lemma~\ref{lem: ray iff at most pi}, and hence $q_i\in A_m$.
\end{proof}

\section{Planes of nonnegative curvature}
\label{sec: nneg}

A key consequence of $G_m\ge 0$ is 
monotonicity of the turn angle and of $\hat\k$. 

\begin{prop}\label{prop: G>0 and monotonicity}
Suppose that $M_m$ has $G_m\ge 0$.
If $0<r_u<r_v$ and $\g_u$ has finite turn angle, 
then $T_{\g_u}\le T_{\g_v}$ with equality if and only if $G_m$ vanishes
on $[r_u, \infty]$.
\end{prop}
\begin{proof} The result is trivial when $G$ is everywhere zero.
Since $\g_u$ has finite turn angle, $m^{-2}$ 
is integrable, and hence 
$m$ is a concave function with $m^\prime>0$ and $m(\infty)=\infty$
by Lemmas~\ref{lem: m growth}.

Set $x:=r_q$, so that the turn angle of $\g_q$
is $\int_x^\infty F_{m(x)}$. As $m^\prime>0$, we can 
change variables by $t:=m(r)/m(x)$ or $r=m^{-1}(tm(x))$ so that
\[
\int_x^\infty F_{m(x)}(r)\,dr=
\int_1^{\frac{m(\infty)}{m(x)}}\frac{dt}{l(t,x)\,t\,\sqrt{t^2-1}}=
\int_1^{\infty}\frac{dt}{l(t,x)\,t\,\sqrt{t^2-1}}
\]
where $l(t,x):=m^\prime(r)$. Computing
\begin{comment}
\[
\frac{\d}{\d x}\frac{1}{l(t,x)}=
-\frac{\frac{\d l(t,x)}{\d x}}{(l(t,x))^2}=
\frac{m^{\prime\prime}(r)}{(m^\prime(r))^2}\,\frac{\d r}{\d x}=
\frac{m^{\prime\prime}(r)\,t\,m^\prime(x)}{(m^\prime(r))^3}>0
\]
\end{comment}
\[
\frac{\d l(t,x)}{\d x}=m^{\prime\prime}(r)\,\frac{\d r}{\d x}=
\frac{m^{\prime\prime}(r)\,t\,m^\prime(x)}{m^\prime(r)}=
-G(r)\,\frac{t\,m^\prime(x)}{m^\prime(r)}\le 0
\]
we see that $l(t,x)$ is non-increasing in $x$. 
Thus if $r_u<r_v$, then $l(t,r_u)\ge l(t,r_v)$ for all $t$ implying
$T_{\g_u}\le T_{\g_v}$. The equality occurs precisely when
$l(t,x)$ is constant on $[1,\infty)\times [r_u, r_v]$, or equivalently,
when $G(m^{-1}(tm(x)))$ vanishes on $[1,\infty)\times [r_u, r_v]$,
which in turn is equivalent to $G=0$ on $[r_u,\infty)$, because
$tm(x)$ takes all values in $(m(r_u), \infty)$ so
$m^{-1}(tm(x))$ takes all values in $(r_u, \infty)$. 
\end{proof}

\begin{lem}
\label{lem: gauss bonnet nonneg}
If $G_m\ge 0$, then $\hat{\k}$ is non-increasing in $r$.
\end{lem}
\begin{proof}
Let $u_1, u_2, v$ be points on $\mu_v$
with $0<r_{u_1}<r_{u_2}<r_{v}$. By Lemma~\ref{lem: approx by segments} 
the ray $\xi_{u_i}$ is the limit of geodesics segments that join 
$u_i$ with points $\tau_{v}(s)$ as $s\to\infty$.
The segments $[u_1,\tau_{v}(s)]$, $[u_2,\tau_{v}(s)]$ 
only intersect at the endpoint $\tau_{v}(s)$ for if they
intersect at a point $z$, then $z$ is a cut point for
$\tau_v(s)$, so $[\tau_{v}(s), u_i]$ cannot be minimizing.
Hence the geodesic triangle
with vertices $u_1$, $v$, $\tau_{v}(s)$ contains
the geodesic triangle with vertices
$u_2$, $v$, $\tau_{v}(s)$. 
Since $G_m\ge 0$, the former triangle has larger
total curvature, which is finite as $M_m$ has finite
total curvature.
As $m$ only vanishes at $0$, concavity of $m$ implies that 
$m$ is non-decreasing. 

If $m$ is unbounded, Clairaut's relation
implies that the angles at $\tau_{v}(s)$ tend to zero as $s\to\infty$.
By the Gauss-Bonnet theorem $\kappa_{\xi_1(0)}-\kappa_{\xi_2(0)}$ 
equals the total curvature of the ``ideal'' triangle
with sides $\xi_1$, $\xi_2$, $[u_1,u_2]$. Thus 
$\hat\k(r_{u_1})\ge\hat\k(r_{u_2})$ with equality if and only if
$G_m$ vanishes on $[r_{u_1}, \infty)$.

If $m$ is bounded, then $\int_1^\infty m^{-2}=\infty$, so
by~\cite[Proposition 1.7]{Tan-ball-poles-ex} 
the only ray emanating from $q$ is $\mu_q$ so that 
$\hat\k=0$ on $M_m\setminus\{o\}$.
For future use note that in this case the angle formed by 
$\mu_q=\xi_q$ and $[q,\tau_q(s)]$ tends to zero as $s\to\infty$, so
Clairaut's relation together with boundedness of $m$
imply that  the angle at $\tau_q(s)$ in the bigon with sides 
$[q,\tau_q(s)]$ and $\tau_q$
also tends to zero as $s\to\infty$.
\end{proof}

\begin{rmk} By the above proof if $G_m\ge 0$ and 
$m^{-2}$ is integrable on $[1, \infty)$, then
$\hat\k(r_{1})=\hat\k(r_{2})$ for some $r_2>r_1$ if and only if 
$G_m$ vanishes on 
$[r_1, \infty)$. 
\end{rmk}

\begin{proof}[Proof of Theorem~\ref{thm-intro: nonneg curv}] 
(i) 
Since rays converge to rays, 
$\Cr_m$ is closed. 
As $o\in\Cr_m$, rotational symmetry and 
Lemma~\ref{lem: gauss bonnet nonneg}
implies that $\Cr_m$ is a closed ball.

(ii) 
Since $m$ is concave and positive, it is non-decreasing, so
$\displaystyle{\liminf_{r\to\infty} m}>0$, and the claim follows from 
Lemma~\ref{lem: crit is o}.

(iii)
We prove the contrapositive that $M_m=\Cr_m$ if and only if
$m^\prime(\infty)\ge \frac{1}{2}$. 
Note that the latter
is equivalent to $c(M_m)\le\pi$, where $c(Z)$ denotes the 
total curvature of a subset $Z\subseteq M_m$ which varies in $[0,2\pi]$.

Suppose $c(M_m)\le\pi$. Fix $q\neq o$, and consider the segments
$[q,\tau_q(s)]$ that by Lemma~\ref{lem: approx by segments}
converge to $\xi_q$ as $s\to\infty$. Consider the bigon
bounded by $[q,\tau_q(s)]$ and its symmetric image under the
reflection that fixes $\tau_q\cup\mu_q$. 
As in the proof of Lemma~\ref{lem: gauss bonnet nonneg}
we see that the angle at $\tau_q(s)$ goes to zero as 
$s\to\infty$, so the sum of angles
in the bigon tends to $2(\pi-\hat{\k}(r_q))$, which 
by the Gauss-Bonnet theorem cannot 
exceed $c(M_m)\le\pi$. We conclude that 
$\hat{\k}(r_q)\ge\frac{\pi}{2}$, so $q\in\Cr_m$.

Conversely, suppose that $\Cr_m=M_m$. Given $\e>0$
find a compact rotationally symmetric
subset $K\subset M_m$ with $c(K)>c(M_m)-\e$.
Fix $q\neq o$ and consider the rays 
$\xi_{\mu_q(s)}$ as $s\to\infty$. If all these rays intersect 
$K$, then they subconverge to a line~\cite[Lemma 6.1.1]{SST},
so by the splitting theorem $M_m$ is the standard $\mathbb R^2$,
and $c(M_m)=0<\pi$. Thus we can assume that there is 
$v$ on the ray $\mu_q$ such that $\xi_{v}$
is disjoint from $K$. Therefore, if $s$ is large enough, then $K$ lies inside the bigon bounded by $[v,\tau_v(s)]$ and its symmetric
image under the reflection that fixes $\tau_q\cup\mu_q$. 
The sum of angles
in the bigon tends to $2(\pi-\hat{\k}(r_{v}))$,
and by the Gauss-Bonnet theorem it is bounded below by $c(K)$.
Since $v\in\Cr_m$, we have $\hat{\k}(r_{v})\ge \frac{\pi}{2}$,
and hence $c(K)\le\pi$. Thus $c(M_m)<\pi+\e$, and since $\e$
is arbitrary, we get $c(M_m)\le\pi$, which completes the 
proof of (iii).

(iv) Since $R_m$ is finite, $m^\prime(\infty)<\frac{1}{2}$ by part (iii).
As $m^\prime(0)=1$, the equation 
$m^\prime(x)=\frac{1}{2}$ has a solution $\r_m$.
As $G_m\ge 0$, the function $m^\prime$ is non-increasing, so
uniqueness of the solution is equivalent to positivity of $G_m(\r_m)$. 
Since $M_m$ is von Mangoldt, $G_m(\r_m)>0$ for otherwise 
$G_m$ would have to vanish for $r\ge \r_m$, implying
$m^\prime(\infty)=m^\prime(\r_m)=\frac{1}{2}$, 
so $R_m$ would be infinite.

Now we show that $\r_m> R_m$. This is clear 
if $R_m=0$ because $\r_m\ge 0$ and 
$m^\prime(0)=1\neq\frac{1}{2}=m^\prime(\r_m)$.
Suppose $R_m>0$. 
Then $m^{-2}$ is integrable by Lemma~\ref{lem: crit is o}, 
so $m^\prime>0$ everywhere by the proof of Lemma~\ref{lem: m growth}.
Hence for any $r_v\ge \r_m$ we have
$m(r_v)\ge m(\r_m)$, which implies $t\,m(r_v)> m(\r_m)$ 
for all $t>1$.
Thus $m^{-1}(t\,m(r_v))> m^{-1}(m(\r_m))=\r_m$.
Applying $m^\prime$ to the inequality, we get in notations
of Proposition~\ref{prop: G>0 and monotonicity} that
$l(t,r_v)<m^\prime (\r_m)=\frac{1}{2}$, 
where the inequality is strict because $G_m(r_m)>0$ by part (iv).
Now (\ref{form: zero curv turn angle}) below implies 
\[
T_{\g_v}=\int_1^{\infty}\frac{dt}{l(t,r_v)\,t\,\sqrt{t^2-1}}>
\int_1^\infty\!\!\frac{2\,dt}{t\,\sqrt{t^2-1}}=
\pi.
\] 
Since $M_m$ is von Mangoldt, $v\notin\Cr_m$
by Lemma~\ref{lem: crit point and ray tangent to parallel}.
In summary, if $r_v\ge \r_m$, then $v\notin\Cr_m$, so
$\r_m> R_m$.

(v) Since $R_m$ is positive and finite, and $M_m$ is von Mangoldt,
there are geodesics 
tangent to parallels whose turn angles are $\le \pi$, and $>\pi$, 
respectively. By Proposition~\ref{prop: G>0 and monotonicity} 
the turn angle is monotone with respect to $r$, so let $r_q$ be the (finite) supremum
of all $x$ such that $\int_x^\infty F_{m(x)}<\pi$. Since $\Cr_m$ is closed,
$q\in\Cr_m$ so that $T_{\g_q}\le \pi$. In fact, 
$T_{\g_q}=\pi$ for if $T_{\g_q}<\pi$, then $r_q$ is not maximal
because by  
Theorems~\ref{thm-intro: away} and \ref{thm: max devia vs turn angle}
the set of points $q$ with $T_{\g_q}<\pi$ is open in $M_m$.
If $G_m(r_q)>0$, then by monotonicity $r_q$ is a unique solution of 
$T_{\g_q}=\pi$. If $G_m(r_q)=0$, then
$G_m\vert_{[r_q,\infty)}=0$ as $M_m$ is von Mangoldt, so 
(\ref{form: zero curv turn angle})
implies that the turn angle of each $\g_v$ with $r_v\ge r_q$ equals
$\frac{\pi}{2m^\prime(r_q)}$. So $m^\prime(r_q)=\frac{1}{2}$
but this case cannot happen as $R_m$ is infinite by (iii).
\end{proof}

In preparation for a proof of 
Theorem~\ref{thm-intro: crit is soul} we
recall that the Cheeger-Gromoll soul construction with basepoint $q$,
described e.g. in~\cite[Theorem V.3.4]{Sak-book}, 
starts by deleting the horoballs associated with
all rays emanating from $q$,
which results in a compact totally convex subset. The next step
is to consider the points of this subset which are at maximal distance
from its boundary, and these points in turn form a compact totally 
convex subset, and after finitely many iterations the process 
terminates in a subset with empty boundary, called a soul. 
As we shall see below, if $G_m\ge 0$, then the soul construction 
with basepoint $q\in \Cr_m\setminus\{o\}$ takes 
no more than two steps; more precisely, deleting
the horoballs for rays emanating from $q$ 
results either in $\{q\}$ or in a segment with $q$
as an endpoint. In the latter case the soul is the midpoint
of the segment.
  
In what follows we let $B_\s$ denote the (open) 
horoball for a ray $\s$ with $\s(0)=q$, i.e.
the union over $t\in [0,\infty)$ 
of the metric balls of radius $t$ centered at $\s(t)$.
Let $H_\s$ denote the complement of $B_\s$ in the 
ambient complete Riemannian manifold.

\begin{lem}
\label{lem: horoball} Let $\s$ be a ray in a complete
Riemannian manifold $M$, and let $q=\s(0)$.
Then for any nonzero $v\in T_q M$ that makes an acute angle with 
$\s$, the point $\exp_q(tv)$ lies in the horoball
$B_\s$ for all small $t>0$.
\end{lem}
\begin{proof} This follows from
the definition of a horoball for if
$\Upsilon$ denotes the image of $t\to \exp_q(tv)$, then 
$\displaystyle{\lim_{s\to +0}\frac{d(\s(s), \Upsilon)}{d(\s(s), q)}}=
\sin \angle(\upsilon^\prime(0), \s^\prime(0)) < 1$,
so $B_\s$ contains a subsegment of $\Upsilon\,\mbox{--}\,\{q\}$ 
that approaches $q$. 
\end{proof}

\begin{proof}[Proof of Theorem~\ref{thm-intro: crit is soul}]
For $q\in\Cr_m$, let $C_q$ denote the complement in $M_m$ of the union of
the horoballs for rays that start at $q$; note that $C_q$ is compact and totally convex. If $C_q$ equals $\{q\}$, then $q$ is a soul. Otherwise, 
$C_q$ has positive dimension and $q\in\d C_q$. 
Set $\g:=\xi_q$; thus $\g$ is a ray.

{\bf Case 1.} Suppose $\frac{\pi}{2}<\hat{\k}(r_q)<\pi$.
Let $\bar\g$ be the clockwise ray that is mapped to $\g$
by the isometry fixing the meridian through $q$.
We next show that $q$ is the intersection of the complements of the
horoballs for rays $\mu_q$, $\g$, $\bar\g$, implying that $q$ is a soul
for the soul construction that starts at $q$. 
As $\k_{\g(0)}>\frac{\pi}{2}$, any nonzero $v\in T_qM_m$ forms
angle $<\frac{\pi}{2}$ with one of
$\mu^\prime(0)$, $\g^\prime(0)$, $\bar{\g}^\prime(0)$,
so $\exp_q(tv)$ cannot lie in the intersection of 
$H_{\mu_q}$, $H_\g$, $H_{\bar\g}$ for small $t$, and
since the intersection is totally convex, it is $\{q\}$.

{\bf Case 2.}
Suppose $\hat{\k}(r_q)=\frac{\pi}{2}$, so that $\g=\g_q$,
and suppose that $G_m$ does not vanish along $\g$. 
By symmetry and Lemma~\ref{lem: horoball}, 
it suffices to show that every point of the segment $[o,q)$ near $q$
lies in $B_\g$. Let $\a$ be the ray from $o$ passing through $q$.
The geodesic $\g$ is orthogonal to $\a$, and it suffices to show
that there is a focal point of $\a$ along $\g$ 
for if $\g(t_0)$ is the first focal point, and $t>t_0$
is close to $t_0$, then there is a variation of $\g\vert_{[0, t]}$
through curves shorter than $\g\vert_{[0, t]}$
that join $\g(t)$ to points of $\a-\{q\}$ near 
$q$~\cite[Lemma III.2.11]{Sak-book}
so these points lie in $B(\g(t), t)\subset B_\g$.

Any $\a$-Jacobi field along $\g$ is of the form $jn$ where
$n$ is a parallel nonzero normal vector field along $\g$
and $j$ solves $j^{\prime\prime}(t)+G_m(r_{\g(t)})j(t)=0$,
$j(0)=1$, $j^\prime(0)=0$. Since $G_m\ge 0$, the function $j$
is concave, so due to its initial values, $j$ must vanish
unless it is constant. The point where $j$ vanishes is focal. 
If $j$ is constant, then $G_m=0$ along $\g$, which is ruled out by assumption.

{\bf Case 3.}
Suppose $\hat{\k}(r_q)=\pi$, i.e. $\g=\tau_q$. 
For any vector $v\in T_q M_m$ pointing inside $C_q$,
for small $t$ the point $\exp_q(tv)$ 
is not in the horoballs for $\mu_q$ and $\tau_q$, and hence $v$
is tangent to a parallel, i.e. $C_q$ is a subsegment of 
the geodesic $\a$ tangent to the parallel through $q$.  
As $C_q$ lies outside the horoballs for $\mu_q$ and $\tau_q$,
these rays there cannot contain focal points of $\a$, implying that
$G_m$ vanishes along $\mu_q$ and $\tau_q$, and hence everywhere,
by rotational symmetry, so that $M_m$ is the standard $\mathbb R^2$,
and $q$ is a soul.

{\bf Case 4.} Suppose $\hat{\k}(r_q)=\frac{\pi}{2}$, 
so that $\g=\g_q$,
and suppose that $G_m$ vanishes along $\g$. 
By rotational symmetry 
$G_m(r)=0$ for $r\ge r_q$, so $m(r)=a(r-r_q)+m(r_q)$ for $r\ge r_q$
where $a>0$, as $m$ only vanishes at $0$. 
The turn angle of $\g$ can be computed explicitly as 
\begin{equation}
\label{form: zero curv turn angle}
\int_{x}^\infty\!\!\!\frac{dr}
{
m(r)\,
\sqrt{\frac{m(r)^2}{m(x)^2}-1}
}=
\int_1^\infty\!\!\frac{dt}{a\, t\,\sqrt{t^2-1}}=
-\frac{1}{a}\,\mathrm{arccot}(\sqrt{t^2-1})\Big{|}_1^\infty=
\frac{\pi}{2a}
\end{equation}
where $x:=r_q$. Since $\g$ is a ray, we deduce that $a\ge \frac{1}{2}$. 

Let $z\le x$ be the smallest number such that 
$m^\prime\vert_{[z,\infty)}=a$; 
thus there is no neighborhood of $z$ in $(0,\infty)$
on which $G_m$ is identically zero. 

Note that $m(r)=a(r-z)+m(z)$ for $r\ge z$, so 
the surface $M_m\,\mbox{--}\, B(o,z)$ is isometric
to $C\,\mbox{--}\, B\left(\bar o, \frac{m(z)}{a}\right)$ where $C$ is the cone with apex $\bar o$
such that cutting $C$ along the meridian from $\bar o$ 
gives a sector in $\mathbb R^2$ of angle $2\pi a$
with the portion inside the radius
$\frac{m(z)}{a}$ removed.

Since $\g_q$ is a ray, Lemma~\ref{lem: horoball} implies the 
existence of a neighborhood $U_q$ of $q$ such that 
each point in $U_p\,\mbox{--}\, [o,q]$ lies in a horoball 
for a ray from $q$.

We now check that $o$ lies in the horoball of $\g_q$. 
Concavity of $m$ implies that the graph of $m$ 
lies below its tangent line at $z$, so evaluating the tangent line
at $r=0$ and using $m(0)=0$ gives $\frac{m(z)}{a}>z$. The Pythagorean
theorem in the sector in $\mathbb R^2$ of angle $2\pi a$ implies that 
\[
d_{M_m}(\g_q(s), o)=\sqrt{s^2+\left(x-z+\frac{m(z)}{a}\right)^2}+
z-\frac{m(z)}{a}
\] 
which is $<s$ for large $s$, implying that $o$
is in the horoball of $\g_q$.

To realize $q$ as a soul, we need to look at 
the soul construction with arbitrary basepoint $v$, which
starts by considering the complement in $M_m$ of the union of the horoballs for all rays from $v$, which by the above is either $v$ or
a segment $[u,v]$ contained in $(o,v]$, where $u$
is uniquely determined by $v$. 
It will be convenient to allow for degenerate segments for which $u=v$;
with this convention the soul is the midpoint of $[u,v]$. 
Since $z$ is the smallest such that $G_m\vert_{[z,\infty)}=0$, 
the focal point argument of Case $2$
shows that $u=v$ when $0<r_v<z$. 
Set $y:=r_v$, and let $e(y):=r_u$; note that $0<e(y)\le y$, and the midpoint of $[u,v]$ has $r$-coordinate $h(y):=\frac{y+e(y)}{2}$. 

To realize each point of $M_m$ as a soul, it suffices to show that
each positive number is in the image of $h$.
Since $h$ approaches zero as $y\to 0$ and approaches infinity as 
$y\to\infty$, it is enough to show that $h$ is continuous and then
apply the intermediate value theorem.

Since $e(y)=y$ when $0<y<z$, we only need to verify continuity of $e$
when $y\ge z$. 
Let $v_i$ be an arbitrary sequence of points on $\a$ converging to $v$,
where as before $\a$ is the ray from $o$ passing through $q$.
Set $v_i:=r_{v_i}$. Arguing by contradiction suppose that $e(y_i)$
does not converge to $e(y)$.
Since $0<e(y_i)\le y_i$ and $y_i\to y$, 
we may pass to a subsequence such that $e(y_i)\to e_\infty\in [0,y]$. 
Pick any $w$ such that $r_w$ lies between $e_\infty$ and $e(y)$.
Thus there is $i_0$ such that 
either $e(y_i)<r_w<e(y)$ for all $i>i_0$, or
$e(y)<r_w<e(y_i)$ for all $i>i_0$.
As $y\ge z$, we know that 
$G_m$ vanishes along $\g_v$, so every $\a$-Jacobi field along $\g_v$ 
is constant. Therefore, 
the rays $\g_{v_i}$ converge uniformly (!) to $\g_v$, as $v_i\to v$,
and hence their Busemann functions $b_i$, $b$ converge pointwise.
Thus $b_i(w)\to b(w)$, but we have chosen $w$
so that $b(w)$, $b_i(w)$ are all nonzero, and
$\mathrm{sign}(b(w))=-\mathrm{sign}(b_i(w))$, which gives a contradiction
proving the theorem.
\end{proof}

\begin{rmk}
In Cases $1, 2, 3$ the soul construction terminates in one step,
namely, if $q\in \Cr_m$, then $\{q\}$ 
is the result of removing the horoballs for all rays 
that start at $q$. We do not know whether the same 
is true in Case $4$ because the basepoint $v$ 
needed to produce the soul $q$ is found implicitly, via an intermediate 
value theorem, and it is unclear how $v$ depends on $q$,
and whether $v=q$.
\end{rmk}

\begin{rmk}\label{rmk: poles when m'=.5}
Let $M_m$ be as in Case $4$ with 
$m^\prime\vert_{[z,\infty)}=\frac{1}{2}$.
If $M_m$ is von Mangoldt, then no point
$q$ with $r_q\ge z$ is a pole because by 
(\ref{form: zero curv turn angle}) the turn angle of $\g_q$ is $\pi$,
which by Theorem~\ref{thm: max devia vs turn angle} cannot happen for a pole. 
\end{rmk}

\section{Smoothed cones made von Mangoldt}
\label{sec: smoothed cones}

\begin{proof}[Proof of Theorem~\ref{thm-intro: slope realized}]
It is of course easy to find a von Mangoldt plane $g_{m_x}$ that has zero curvature
near infinity, but prescribing the slope of $m^\prime$ there takes more effort.
We exclude the trivial case $x=1$ in which $m(r)=r$ works.

For $u\in [0,\frac{1}{4}]$ set $K_u(r)=\frac{1}{4(r+1)^2}-u$, and let
$m_u$ be the unique solution of (\ref{form: ode}) with $K=K_u$. Then
$g_{m_u}$ is von Mangoldt. For $u>0$ let $z_u\in [0,\infty)$ be the unique zero  
of $K_u$; note that $z_u$ is the global
minimum of $m_u^\prime$, and $z_u\to\infty$ as $u\to 0$.

\begin{lem} 
\label{lem: in prescribed slope}
The function $u\to m^\prime_u(z_u)$ 
takes every value in $(0,1)$ as $u$ varies in $(0,\frac{1}{4})$.
\end{lem}
\begin{proof}
One verifies that $m_0(r)=\ln(r+1)\,\sqrt{r+1}$, i.e. the right hand side
solves (\ref{form: ode}) with $K=K_0$.
Then $m_0^\prime=\frac{2+\ln(r+1)}{2\sqrt{r+1}}$ is a positive function
converging to zero as $r\to \infty$.
By Sturm comparison $m_u\ge m_0>0$
and $m_u^\prime\ge m_0^\prime>0$. 

We now show that $m_u^\prime(z_u)\to 0$ as $u\to +0$.
To this end fix an arbitrary $\e>0$. 
Fix $t_\e$ such that $m_0^\prime(t_\e)<\e$.
By continuous dependence on parameters 
$(m_u, m^\prime_u)$ converges to $(m_0, m_0^\prime)$ uniformly
on compact sets as $u\to 0$. So for all small $u$ 
we have $m_u^\prime(t_\e)<\e$ and also $t_\e<z_u$. 
Since $m^\prime_u$ decreases on $(0,z_u)$,
we conclude that $0<m_u^\prime(z_u)<m_u^\prime(t_\e)<\e$,
proving that $m_u^\prime(z_u)\to 0$ as $u\to +0$.

On the other hand, $m_\frac{1}{4}^\prime(z_\frac{1}{4})=1$
because $z_\frac{1}{4}=0$ and by the initial condition
$m_\frac{1}{4}^\prime(0)=1$. 
Finally, the assertion of the lemma follows from
continuity of the map $u\to m_u^\prime(z_u)$, because then
it takes every value within $(0,1)$ as $u$ varies in $(0,\frac{1}{4})$. 
(To check continuity of the map fix $u_*$, take an arbitrary $u\to u_*$
and note that $z_u\to z_{u_*}$, so since $m^\prime_u$ converges
to $m_{u_*}^\prime$ on compact subsets, it does so on a neighborhood of $z_{u_*}$,
so $m_u^\prime(z_u)$ converges to $m_{u_*}^\prime(z_{u_*})$).
\end{proof}

Continuing the proof of the theorem,
fix an arbitrary $u>0$.
The continuous function $\max(K_u, 0)$ is decreasing and smooth on $[0,z_u]$
and equal to zero on $[z_u,\infty)$. So 
there is a family of non-increasing smooth functions $G_{u,\e}$
depending on small parameter $\e$
such that $G_{u,\e}=\max(K_u, 0)$ outside the $\e$-neighborhood of $z_u$.
Let $m_{u,\e}$ be the unique solution of (\ref{form: ode}) with $K=G_{u,\e}$;
thus $m^\prime_{u,\e}(r)=m^\prime_{u,\e}(z_u+\e)$ for all $r\ge z_u+\e$.
If $\e$ is small enough, then $G_{u,\e}\le K_0$, so $m_{u,\e}\ge m_0>0$ and 
$m^\prime_{u,\e}\ge m^\prime_0>0$.
By continuous dependence on parameters, 
the function $(u,\e)\to m^\prime_{u,\e}$ is continuous, and moreover
$m^\prime_{u,\e}(z_u+\e)\to m^\prime_{u}(z_u)$ as $\e\to 0$, and $u$
is fixed.

Fix $x\in (0,1)$. By Lemma~\ref{lem: in prescribed slope} 
there are positive $v_1$, $v_2$ 
such that 
$m_{v_1}^\prime(z_{v_1})<x<m_{v_2}^\prime(z_{v_2})$.
Letting $u$ of the previous paragraph to be $v_1$, $v_2$, we find $\e$
such that $m^\prime_{v_1,\e}(z_{v_1}+\e)<x<m^\prime_{v_2,\e}(z_{v_2}+\e)$,
so by the intermediate value theorem there is $u$ with 
$m^\prime_{u,\e}(z_u+\e)=x$. Then the metric $g_{m_{u,\e}}$ has
the asserted properties for $\r=z_u+\e$.
\end{proof}

\section{Other applications}
\label{sec: other proofs}

%Theorems~\ref{thm-intro: nonneg curv} and \ref{thm-intro: crit is soul}
%are proved in Section~\ref{sec: nneg}, and a proof of
%Theorem~\ref{thm-intro: slope realized} is in Section~\ref{sec: smoothed 
%cones}.

\begin{proof}[Proof of Lemma~\ref{lem-intro: crit}]
Assuming $\hat r(\hat q)\notin r(\Cr_m)$ we will show that $\hat q$
is not a critical point of $\hat r$.
Since $\hat M$ is complete and noncompact, there is a ray $\hat\g$ emanating
from $\hat q$.
Consider the comparison triangle
$\Delta(o, q, q_i)$ in $M_m$ for any geodesic triangle with vertices 
$\hat o$, $\hat q$, $\hat\g(i)$. Passing to a subsequence, arrange
so that the segments $[q, q_i]$ subconverge to a ray, which we denote by $\g$.
Since $q\notin \Cr_m$, the angle formed by $\g$ and $[q,o]$ is $>\frac{\pi}{2}$,
and hence for large $i$
the same is true for the angles formed by $[q,q_i]$ and $[q,o]$.  
By comparison, $\hat\g$ forms angle $>\frac{\pi}{2}$ with
any segment joining $\hat q$ to $\hat o$, i.e. $\hat q$ is not
a critical point of $\hat r$.
\end{proof}

\begin{proof}[Proof of Theorem~\ref{intro-thm: poles}]
(a) 
Let $P_m$ denote the set of poles; 
it is a closed metric ball~\cite[Lemma 1.1]{Tan-cut}.
Moreover, $P_m$ clearly lies in the connected component 
$A_m^o$ of $A_m\cup\{o\}$ that contains $o$, and hence in the component
of $\Cr_m$ that contains $o$.
By Theorem~\ref{thm-intro: away}
$A_m$ is open in $M_m$, so $A_m\cup\{o\}$ is locally path-connected,
and hence $A_m^o$ is open in $M_m$. 
If $P_m$ were equal to $A_m^o$, the latter would be closed,
implying $A_m^o=M_m$, which is impossible as the ball has finite radius.

(b) The "if'' direction is trivial as $P_m\subset\Cr_m$. Conversely,
if $\Cr_m\neq\{o\}$, then by Lemma~\ref{lem: crit is o}
$m^{-2}$ is integrable and 
$\underset{r\to\infty}{\liminf}\, m(r)>0$, so
$R_p>0$~\cite{Tan-ball-poles-ex}. 
\end{proof}

\begin{proof}[Proof of Theorem~\ref{thm-intro: neg curv}]
By assumption there is a point of negative curvature, and since
the curvature is non-increasing, outside a compact subset the curvature is
bounded above by a negative constant. As $\displaystyle{\liminf_{r\to\infty} m(r)}>0$, 
$m$ is bounded below by a positive constant outside any neighborhood of $0$, 
so $\int_0^\infty m=\infty$. 
Hence the total curvature 
$2\pi\int_0^\infty G_m(r)\,m(r)\,dr$ is $-\infty$.

Hence there is a metric ball $B$ of finite positive
radius centered at $o$ such that the total curvature of $B$
is negative, and such that no point of $G_m\ge 0$ lies outside $B$.
By~\cite[Theorem 6.1.1, page 190]{SST}, for any $q\in M_m$ 
the total curvature of the set
obtained from $M_m$ by removing all rays that start at $q$ is in $[0,2\pi]$.
So for any $q$ there is a ray that starts at $q$ and intersects $B$. 
  
If $q$ is not in $B$, then
the ray points away from infinity, so $q\in A_m$ and
any point on this ray is in $\Cr_m$.
Thus $M_m\,\mbox{--}\, A_m$ lies in $B$. 
Since $\Cr_m\neq\{o\}$, Theorem~\ref{intro-thm: poles} implies that $R_p>0$.
Letting $q$ run to infinity the rays subconverge to a line that intersects $B$
(see e.g.~\cite[Lemma 6.1.1, page 187]{SST}. 

If $m^\prime(r_p)=0$, then the parallel through $p$ is a geodesic but not a
ray, so Lemma~\ref{lem: crit point and ray tangent to parallel}
implies that no point on the parallel through $p$ is in $\Cr_m$. Since
$\Cr_m$ contains $o$ and all points outside a compact set, 
$\Cr_m$ is not connected; the same argument proves that 
$A_m$ is not connected.
\end{proof}

\begin{ex}
\label{ex: m' vanishes}
Here we modify~\cite[Example 4]{Tan-cut} to construct a von Mangoldt
plane $M_m$ such that $m^\prime$ has a zero, 
and neither $A_m$ nor $\Cr_m$ is connected.
Given $a\in (\frac{\pi}{2}, \pi)$
let $m_0(r)=\sin r$ for $r\in [0,a]$, and define $m_0$ for $r\ge a$ so that $m_0$
is smooth, positive, and $\inf\lim_{r\to\infty}\,m_0>0$. 
Thus $K_0:=-\frac{m_0^{\prime\prime}}{m_0}$ equals $1$ on $[0,a]$. 
Let $K$ be any smooth non-increasing function with $K\le K_0$ and $K\vert_{[0,a]}=1$.
Let $m$ be the solution of (\ref{form: ode});
note that $m(r)=\sin(r)$ for $r\in [0,a]$ so that $m^\prime$ vanishes 
at $\frac{\pi}{2}$.
By Sturm comparison $m\ge m_0>0$, and hence $M_m$ is a von Mangoldt plane.
Since $m^\prime(a)<0$ and $m>0$ for all $r>0$, 
the function $m$ cannot be concave, so $K=G_m$ eventually becomes negative, 
and Theorem~~\ref{thm-intro: neg curv} implies that $A_m$ and $\Cr_m$
are not connected.
\end{ex}

\begin{ex}
\label{ex: m' positive but non-connected}
Here we construct a von Mangoldt plane such that $m^\prime>0$ everywhere
but $A_m$ and $\Cr_m$ are not connected.
Let $M_{n}$ be a von Mangoldt plane such that $G_n\ge 0$ and 
$n^\prime>0$ everywhere, and $R_n$ is finite 
(where $R_n$ is the radius of the ball $\Cr_n$).
This happens e.g. for any paraboloid, any
two-sheeted hyperboloid with $n^\prime(\infty)<\frac{1}{2}$, or
any plane constructed in Theorem~\ref{thm-intro: slope realized}
with $n^\prime(\infty)<\frac{1}{2}$.
Fix $q\notin \Cr_n$. Then $\g_q$ has turn angle $>\pi$, so there is
$R>r_q$ such that $\int_{r_q}^R F_{n(r_q)}>\pi$. Let $G$ be any smooth
non-increasing function such that $G=G_n$ on $[0,R]$ and $G(z)<0$ for
some $z>R$. Let $m$ be the solution of (\ref{form: ode}) with $K=G$.
By Sturm comparison $m\ge n>0$ and $m^\prime\ge n^\prime>0$
everywhere; see Remark~\ref{rmk: sturm derivative}. 
Since $m=n$ on $[0,R]$, on this interval 
we have $F_{m(r_q)}=F_{n(r_q)}$, so in the von Mangoldt plane $M_m$
the geodesic $\g_q$ has turn angle $>\pi$, which implies that no
point on the parallel through $q$ is in $\Cr_m$. 
Now parts (3)-(4) of Theorem~~\ref{thm-intro: neg curv} 
imply that $A_m$ and $\Cr_m$
are not connected.
\end{ex}

\begin{thm} 
\label{thm: neck}
Let $M_m$ be a von Mangoldt plane such that 
$m^\prime\vert_{[0,y]}>0$ and $m^\prime\vert_{[x,y]}<\frac{1}{2}$. 
Set $f_{m,x}(y):=m^{-1}\left(\cos(\pi b)\, m(y)\right)$, 
where $b$ is the maximum of $m^\prime$ on $[x,y]$.
If $x\le f_{m,x}(y)$, then $r(\Cr_m)$ and $[x, f_{m,x}(y)]$ are disjoint. 
\end{thm}

\begin{proof} Set $f:=f_{m,x}$.
Arguing by contradiction assume there is 
$q\in \Cr_m$ with $r_q\in [x, f(y)]$. Then $\g_q$ has turn angle 
$\le \pi$, so if $c:=m(r_q)$, then 
\[
\pi\ge\int_{r_q}^{\infty}\frac{c\,dr}{m\,\sqrt{m^2-c^2}}>
\int_{r_q}^{y}\frac{c\,dr}{m\,\sqrt{m^2-c^2}}=
\int_c^{m(y)}\!\!\!\! \frac{c\,dm}{m^\prime(r)\,m\,\sqrt{m^2-c^2}}
\ge 
\]
\[
\int_c^{m(y)}\!\!\!\! \frac{c\,dm}{b\,m\,\sqrt{m^2-c^2}}=
\frac{1}{b}\,\mathrm{arccos}\left(\frac{c}{m(y)}\right)
\]
so that $\pi b>\mathrm{arccos}\left(\frac{c}{m(y)}\right)$, 
which is equivalent to 
$\cos(\pi b)\, m(y)< m(r_q)$.

On the other hand, $m(f(y))$ is in the interval 
$[0,m(y)]$ on which $m^{-1}$ is increasing, so $f(y)<y$, 
and therefore $m$ is increasing on $[x, f(y)]$.
Hence $r_q<f (y)$ implies $m(r_q)<m(f(y))=\cos(\pi b)\, m(y)$,
which is a contradiction.
\end{proof}

\begin{proof}[Proof of Theorem~\ref{intro-thm: neck nneg}]
We use notation of Theorem~\ref{thm: neck}.
The assumptions on $n$ imply $n^\prime>0$, 
$n^\prime\vert_{[x,\infty)}<\frac{1}{2}$, and
$b=n^\prime(x)$. Hence $f_{n,x}$ is an increasing 
smooth function of $y$ with $f_{n,x}(\infty)=\infty$.
In particular, if $y$ is large enough, then $f_{n,x}(y)>z>x$;
fix $y$ that satisfies the inequality.
Now if $M_m$ is any von Mangoldt plane with $m=n$ on $[0,y]$, 
then $f_{m,x}(y)=f_{n,x}(y)$, so $M_m$ satisfies the assumptions  
of Theorem~\ref{thm: neck}, so $[x,z]$ and $r(\Cr_m)$ are disjoint. 
\end{proof}

\appendix

\section{Von Mangoldt planes}
\label{sec: vm planes} 

The purpose of this appendix is to 
discuss what makes von Mangoldt planes special
among arbitrary rotationally symmetric planes.

For a smooth function $m\co [0,\infty)\to [0,\infty)$ 
whose only zero is $0$, let $g_m$ denote the rotationally 
symmetric inner product on the tangent bundle to
$\mathbb R^2$ that equals the 
standard Euclidean inner product at the origin
and elsewhere is given in polar coordinates by $dr^2+m(r)^2d\th^2$.
It is well-known (see e.g.~\cite[Section 7.1]{SST}) that 
\begin{itemize}
\item
any rotationally symmetric complete smooth Riemannian metric on
$\mathbb R^2$ is isometric to some $g_m$; as before $M_m$ denotes
$(\mathbb R^2, g_m)$;
\item
if $\bar m\co\mathbb R\to\mathbb R$ denotes the unique odd function
such that $\bar m\vert_{[0,\infty)}=m$, then
$g_m$ is a smooth Riemannian metric on
$\mathbb R^2$ if and only if $m^\prime(0)=1$ and $\bar m$ is smooth;
\item 
if $g_m$ is a smooth metric on $\mathbb R^2$, then $g_m$ is complete, 
and the sectional curvature 
of $g_m$ is a smooth function on $[0,\infty)$ that equals 
$-\frac{ m^{\prime\prime}}{m}$. 
\end{itemize}

It is easier to visualize $M_m$ as a surface of revolution in $\mathbb R^3$, 
so we recall:

\begin{lem}\ \newline
\textup{(1)} $M_m$ is isometric to a surface of revolution
in $\mathbb R^3$ if and only if $|m^\prime|\le 1$.\newline
\textup{(2)}
$M_m$ is isometric to a surface of revolution
$(r\cos\phi,\, r\sin\phi,\, g(r))$ in $\mathbb R^3$
if and only if $0<m^\prime\le 1$.
\end{lem}
\begin{proof}
(1)
Consider a unit speed curve $s\to (x(s), 0, z(s))$ 
in $\mathbb R^3$ where $x(s)\ge 0$ and $s\ge 0$.
Rotating the curve about the $z$-axis gives the surface
of revolution \[
(x(s)\cos\phi, x(s)\sin\phi, z(s))
\]
with metric $ds^2+x(s)^2d\phi^2$. For this metric to be equal
to $ds^2+m(s)^2d\phi^2$ we must have $m(s)=x(s)$.
Since the curve has unit speed, $|x^\prime(s)|\le 1$, so
a necessary condition for writing the metric as a surface of revolution 
is $|m^\prime(s)|\le 1$. 
It is also sufficient for if $|m^\prime(s)|\le 1$,
then we could let $z(s):=\int_0^s\sqrt{1-(m^\prime(s))^2}ds$,
so that now $(m(s), z(s))$ has unit speed. 

(2)
If furthermore $m^\prime>0$ for all $s$, then the inverse
function of $m(s)$ makes sense,
and we can write the surface of revolution 
$(m(s)\cos\phi,\, m(s)\sin\phi,\, z(s))$ as
$(x\cos\phi,\, x\sin\phi,\, g(x))$ where $x:=m(s)$ and $g(x):=z(m^{-1}(x))$.
Conversely, given the surface $(x\cos\phi,\, x\sin\phi,\, g(x))$, the 
orientation-preserving arclength parametrization $x=x(s)$ of the curve
$(x,0,g(x))$ satisfies $x^\prime>0$.
%one has $1=(x^\prime(s))^2(1+(g^\prime(x(s))^2)$ so $x^\prime(s)\neq 0$.
%Since $x=f$, we know $x(0)=0$ and $x^\prime(0)=1$ so we get.
\end{proof}

\begin{ex}
The standard $\mathbb R^2$ is the only von Mangoldt plane with $G_m\le 0$ that
can be embedded into $\mathbb R^3$ as a surface of revolution because
$m^\prime(0)=1$ and $m^\prime$ is non-decreasing afterwards.
%For the hyperbolic plane $m(s)=\sinh(s)$, so this
%von Mangoldt plane cannot be embedded as a 
%surface of revolution in $\mathbb R^3$, and
%actually, it cannot be embedded in $\mathbb R^3$ 
%at all by Hilbert's theorem. 
\end{ex}

\begin{ex} 
\label{ex: nneg m' append}
If $G_m\ge 0$, then $m^\prime\in [0,1]$
because $m>0$, $m^\prime$ is non-increasing, and $m^\prime(0)=1$,
so that $M_m$ is isometric to a surface of revolution in $\mathbb R^3$.
In fact, if $m^\prime(s_0)=0$, then $m\vert_{[s_0,\infty)}=m(s_0)$,
i.e. outside the $s_0$-ball about the origin $M_m$ is a cylinder.
Thus except for such surfaces $M_m$ can be written as 
$(x\cos\phi, x\sin\phi, g(x))$ for
$g(x)=\int_0^{m^{-1}(x)}\sqrt{1-(m^\prime(s))^2}ds$.
Paraboloids and two-sheeted hyperboloids are von Mangoldt planes
of positive curvature~\cite[pp. 234-235]{SST} and they are of the form
$(x\cos\phi, x\sin\phi, g(x))$.
\end{ex}

The defining property $G_m^\prime\le 0$ 
of von Mangoldt planes clearly restricts the behavior of $m^\prime$. 
Let $Z(G_m)$ denote the set where $G_m$ vanishes; as
$M_m$ is von Mangoldt, $Z(G_m)$ is closed and connected, and
hence it could be equal to the empty set, a point, or an interval,
while $m^\prime$ behaves as follows.
\begin{itemize}
\item[(i)] If $G_m>0$, then $m^\prime$ is decreasing
and takes values in $(0,1]$.
\item[(ii)] If $G_m\le 0$, then $m^\prime$ is non-decreasing
and takes values in $[1, \infty)$. 
\item[(iii)] 
If $Z(G_m)$ is a positive number $z$, 
then $m^\prime$ decreases on $[0,z)$ and increases on $(z,\infty)$, 
and $m^\prime$ may have two, one, or no zeros. 
\item[(iv)] If $Z(G_m)=[a, b]\subset (0,\infty]$, 
then $m^\prime$ decreases on $[0,a)$,
is constant on $[a,b]$, and increases on $(b,\infty)$ 
if $b<\infty$. Also either $m^\prime\vert_{[a,b]}=0$
or else $m^\prime$ has two, or no zeros.
\end{itemize}

\begin{rmk}
All the above possibilities occur with one possible
exception: in cases (iii)-(iv) we are not aware 
of examples where $m^\prime$ vanishes
on $Z(G_m)$.
\end{rmk}

\begin{rmk} 
\label{rmk: app, tot curv}
Thus if $M_m$ is von Mangoldt, then
$m^\prime$ is monotone near infinity,
so $m^\prime(\infty)$ exists; moreover, 
$m^\prime(\infty)\in [0,\infty]$, for otherwise 
$m$ would vanish on $(0,\infty)$. 
It follows that $M_m$ admits total curvature, which equals
\[
\int_0^{2\pi}\int_0^\infty G_m\, m\,dr\, d\th=
-2\pi\int_0^\infty m^{\prime\prime}=
2\pi(1-m^\prime(\infty))\,\in [-\infty, 2\pi].
\] 
Here the {\it total curvature of a subset\,} $A\subset M_m$
is the integral of $G_m$ over $A$ with respect to
the Riemannian area form $m\,dr\, d\th$, provided the integral converges
to a number in $[-\infty,\infty]$,
in which case we say that $A$ {\it admits total curvature\,}.
\end{rmk}

\begin{rmk}
The zeros of $m^\prime$ correspond to parallels that are geodesics
and are of interest. In contrast with restrictions on the zero set 
of $m^\prime$ for von Mangoldt planes, 
if $M_m$ is not necessarily von Mangoldt, then
any closed subset of $[0,\infty)$ that does not contain $0$
can be realized as the set of zeros of $m^\prime$. 
(Indeed, for any closed subset of a manifold there is a 
smooth nonnegative function that vanishes precisely on the 
subset~\cite[Whitney's Theorem 14.1]{BroJan}. It follows that
if $C$ is a closed subset of $[0,\infty)$ that does not contain $0$,
then there is a smooth function $g\co [0,\infty)\to [0,\infty)$ that 
is even at $0$, satisfies $g(0)=1$,
and is such that $g(s)=0$ if and only if $s\in C$. 
If $m$ is the solution of $m^\prime=g$, $m(0)=0$;
then $M_m$ has the promised property).
\end{rmk}

A common way of constructing von Mangoldt planes 
involves the Jacobi initial value problem
\begin{equation}
\label{form: ode}
m^{\prime\prime}+Km=0,\ \ m(0)=0,\ \ m^\prime(0)=1
\end{equation}
where $K$ is smooth on $[0,\infty)$.
It follows from the proof of~\cite[Lemma 4.4]{KazWar-open} that 
$g_m$ is a complete smooth
Riemannian metric on $\mathbb R^2$ if and only if
the following condition holds
\begin{center}
\hspace{-44pt}($\star$)\qquad\qquad\it 
the (unique) solution $m$ of \textup{(\ref{form: ode})} 
is positive on $(0,\infty)$.\rm 
\end{center}

\begin{rmk} 
\label{rmk: sturm}
A basic tool that produces solutions of (\ref{form: ode})
satisfying condition ($\star$) is the Sturm comparison theorem 
that implies that 
%if $m_i$ solves (\ref{form: ode}) with $K=K_i$ 
%where $i\in\{1,2\}$, and if $K_2\le K_1$, then $m_2\ge m_1$. 
if $m_1$ is a positive function that solves (\ref{form: ode}) with $K=K_1$, 
and if $K_2$ is any non-increasing smooth function with 
$K_2\le K_1$, then the solution $m_2$ of (\ref{form: ode}) with $K=K_2$
satisfies $m_2\ge m_1$, so that $g_{m_2}$ is a von Mangoldt plane.  
\end{rmk}

\begin{ex}
If $K$ is a smooth function on $[0, \infty)$
such that $\max(K,0)$ has compact support, then
a positive multiple of $K$ can be realized as 
the curvature $G_m$ of some $M_m$; of course, 
if $K$ is non-increasing, then $M_m$ is von Mangoldt.
(Indeed, in~\cite[Lemma 4.3]{KazWar-open} Sturm comparison was
used to show that if 
$\int_t^{\infty}\max(K, 0)\le\frac{1}{4t+4}$ for all $t\ge 0$, then
$K$ satisfies ($\star$), and in particular, if
$\max(K,0)$ has compact support, then there is a constant $\e>0$
such that the above inequality holds for $\e K$). 
\end{ex}

\begin{rmk}
\label{rmk: sturm derivative}
A useful addendum to Remark~\ref{rmk: sturm}
is that the additional assumption $m_1^\prime\ge 0$
implies $m_2^\prime\ge m_1^\prime>0$.
(Indeed, the function $m_1^\prime\, m_2-m_1\,m_2^\prime$ 
vanishes at $0$ and has nonpositive derivative $(-K_1+K_2)\,m_1\,m_2$,
so $m_1^\prime\, m_2\le m_1\,m_2^\prime$. As $m_1$, $m_2$, $m_1^\prime$
are nonnegative, so is  $m_2^\prime$. Hence,
$m_1\,m_2^\prime\le m_2\,m_2^\prime$, which gives
$m_1^\prime\, m_2\le m_2\,m_2^\prime$, and the claim follows by
canceling $m_2$).
\end{rmk}

\begin{quest} Let $m_0\co [r_0,\infty)\to (0,\infty)$ be a smooth function 
such that $r_0>0$ and $-\frac{m_0^{\prime\prime}}{m_0}$ is non-increasing.
What are sufficient conditions for (or obstructions to) extending
$m_0$ to a function $m$ on $[0,\infty)$ such that $g_m$ 
is a von Mangoldt plane?
\end{quest}

\section{A calculus lemma}
\label{sec: calc lem}

This appendix contains an elementary lemma on continuity
and differentiability of the turn angle, which is needed
for Theorem~\ref{thm: max devia vs turn angle}.

Given numbers $r_q>r_0>0$, 
let $m$ be a smooth self-map of $(0,\infty)$ such that
\begin{itemize}
\item $m^\prime>0$ on $[r_0, r_q]$, 
\item $m(r)>m(r_q)$ for $r>r_q$, 
\item $m^{-2}$ is integrable on $(1,\infty)$,
\item $\displaystyle{\liminf_{r\to\infty}\, m(r)}>m(r_q)$. 
\end{itemize}

\begin{ex} Suppose $G_m\ge 0$ or $G_m^\prime\le 0$.
If $\g_q$ is a ray on $M_m$, and $r_0$ is sufficiently close to $r_q$,
then $m$ satisfies the above properties by 
Lemmas~\ref{lem: basic choke}, 
\ref{ex: ray exists implies properties of m}, \ref{lem: m growth}.
\end{ex}

Set $c_0:=m(r_0)$ and $c_q:=m(r_q)$. 
Let $T=T(c)$ be the function given by the integral 
(\ref{form: upward turn angle via c}) for $c=c_q$, and by the sum of integrals
(\ref{form: turn angle via c}) for $c_0\le c\le c_q$, where $F_c$ is given by 
(\ref{form: defn F_c}) and $r_u:=m^{-1}(c)$, 
where $m^{-1}$ is the inverse of $m\vert_{[r_0, r_q]}$. 

\begin{lem} 
\label{lem: calc}
Under the assumptions of the previous paragraph,
$T$ is continuous on $(c_0, c_q]$, continuously differentiable on 
$(c_0, c_q)$, and $T^\prime(c)\sqrt{c_q^2-c^2}$ 
converges to $-\frac{1}{m^\prime(r_q)}<0$ as $c\to c_q-$.
\end{lem}
\begin{proof} By definition $T$ equals $\int_{r_q}^\infty F_c+
\int_{r_u}^{r_q} F_c$ if $c\in [c_0, c_q)$ and 
$T=\int_{r_q}^\infty F_{c}$ if $c=c_q$. 
Step 1 shows that 
$\int_{r_q}^\infty F_c$ depends continuously on 
$c\in [c_0, c_q]$, 
while Step 2 establishes continuity of $T$ at $c_q$.
In Steps 3--4 we prove continuous differentiability and
compute the derivatives of the integrals
$\int_{r_q}^\infty F_c$, $\int_{r_u}^{r_q} F_c$ 
with respect to $c\in (c_0, c_q)$.
Step 5 investigates the behaviour of $T^\prime(c)$ 
as $c\to c_q$.

Recall that the integral 
$\int_a^b H_c(r)dr$ depends continuously on $c$ 
if for each $r\in (a,b)$ the map $c\to H_c(r)$ is continuous, 
and every $c$ has a neighborhood $U_0$ in which
$|H_c|\le h_0$ for some integrable function $h_0$.
If in addition each map $c\to H_c(r)$ is $C^1$, and 
every $c$ has a neighborhood $U_1$ where 
$|\frac{\d H_c}{\d c}|\le h_1$ for an integrable function $h_1$,
then $\int_a^b H_c(r)dr$ is $C^1$ and differentiation
under the integral sign is valid; the same conclusion holds
when $H_c$ and $\frac{\d H_c}{\d c}$ are continuous in the closure of
$U_1\times (a,b)$.

{\bf Step 1.}\ \
The integrand $F_c$ is smooth over $(r_u,\infty)$, because 
the assumptions on $m$ imply that 
$m(r)> c$ for $r> r_u$.

Since $0<c\le c_q$ we have $F_c\le F_{c_q}=\frac{c_q}{m\sqrt{m^2-c_q^2}}$
which is integrable on $(r_q, \infty)$. Indeed, fix $\de>r_q$ and note that
since $m^{-2}$ is integrable on $(\de,\infty)$, so is $F_{c_q}$.
To prove integrability of $F_{c_q}$ on $(r_q, \de)$,
note that $h(r):=\frac{m(r)-m(r_q)}{r-r_q}$
is positive on $[r_q, \infty)$, as $h(r_q)=m^\prime(r_q)>0$ 
and $m(r)>m(r_q)$ for $r>r_q$. Then $F_{c_q}$ is the product of
$(r-r_q)^{-1/2}$ and a function that is smooth on $[r_q, \de]$,
and hence $F_{c_q}$ is integrable on $(r_q, \de)$.

Thus the integrals $\int_{r_q}^{\de} F_{c}(r) dr$ and
$\int_{\de}^{\infty} F_{c}(r) dr$ depend continuously on $c\in (0,c_q]$,
and hence so does their sum $\int_{r_q}^{\infty} F_{c}(r) dr$.

{\bf Step 2.}\ \
As $c\to c_q$, the integral $\int_{r_u}^{r_q} F_c$ converges to zero, 
for if $K$ is the maximum of $(m\, m^\prime\, \sqrt{m+c})^{-1}$
over the points with $r\in [r_0, r_q]$ and $c\in [c_0, c_q]$, then 
\[
\int_{r_u}^{r_q} F_c\le K\int_{r_u}^{r_q} \frac{m^\prime dr}{\sqrt{m-c}}=
K\int_0^{c_q-c}\frac{dt}{\sqrt{t}}
\]
which goes to zero as $c\to c_q$. Thus $T$ is continuous at $c=c_q$.

{\bf Step 3.}\ \
To find an integrable function dominating $\frac{\d F_c}{\d c}$ 
on $(r_q,\infty)$ locally in $c$, note that every
$c\in (c_0,c_q)$ has a neighborhood of the form $(c_0,  c_q-\de)$
with $\de>0$, and over this neighborhood
\[
\frac{\d  F_c}{\d c}=
%\frac{\d}{\d c}\frac{1}{m\sqrt{\frac{m^2}{c^2}-1}}=
\frac{m}{(m^2-c^2)^{3/2}}\le 
\frac{m}{(m^2-(c_q-\de)^2)^{3/2}},
\]
where the right hand side
is integrable over $[r_q,\infty)$, as $m^{-2}$ is integrable
at $\infty$; thus
\[
\frac{d}{dc} \int_{r_q}^\infty F_c=\int_{r_q}^\infty 
\frac{m}{(m^2-c^2)^{3/2}} dr
\]
is continuous with respect to $c\in (c_0,c_q)$. 
This integral diverges if $c=m(r_q)$.

{\bf Step 4.}\ \
To check continuity of $\int_{r_u}^{r_q} F_c$ 
change variables via $t:=\frac{m}{c}$ so that $r=m^{-1}(tc)$. 
Thus $dt=m^\prime(r)\frac{dr}{c}=n(tc)\frac{dr}{c}$ where
$n(r):=m^\prime(m^{-1}(r))$, and
\[
\int_{r_u}^{r_q} F_{c}(r) dr=\int_1^{c_q/c}{\bar F}_c(t) dt
\qquad\text{where}\qquad {\bar F}_c(t)=\frac{1}{n(tc)\,t\,\sqrt{t^2-1}}.
\]
Since $m^\prime>0$ on $[r_0, r_q]$ and $n(tc)=m^\prime(r)$, 
the function ${\bar F}_c$ is smooth
over $(1, \frac{c_q}{c})$. To prove continuity of
$\int_1^{c_q/c}{\bar F}_c$, fix an arbitrary 
$(u, v)\subset (c_0, c_q)$. If $c\in (u,v)$
and $t\in (1,\frac{c_q}{c})$, then $m^{-1}(tc)$ lies
in the $m^{-1}$-image of $(u,\frac{v}{u}c_q)$, which by taking the interval
$(u,v)$ sufficiently
small can be made to lie in an arbitrarily small neighborhood of
$[r_0, r_q]$, so we may assume that $m^\prime>0$ on that neighborhood.
It follows that the maximum $K$ of $\frac{1}{n(tc)}$ 
over $c\in [u,v]$ and $t\in [1,\frac{c_q}{c}]$ is finite, and
$|{\bar F}_c|\le \frac{K}{t\,\sqrt{t^2-1}}$ for $c\in (u,v)$,
i.e. $|F_c|$ is locally dominated by an integrable function
that is independent of $c$; for the same reason 
the conclusion also holds for $\frac{\d {\bar F}_c}{\d c}=
-\frac{n^\prime(tc)}{n(tc)^2\sqrt{t^2-1}}$.

Finally, given $c_*\in (c_0, c_q)$ fix $\de\in (1, \frac{c_q}{c_*})$, 
and write $\int_1^{c_q/c}{\bar F}_c=
\int_1^{\de}{\bar F}_c+\int_\de^{c_q/c}{\bar F}_c$ for $c$ varying near $c_*$.
The first summand is $C^1$ at $c_*$, as the integrand and its derivative are
dominated by the integrable function near $c_*$.
The second summand is also $C^1$ at $c_*$ as the integrand is $C^1$ on 
a neighborhood of $\{c_*\}\times [\de, \frac{c_q}{c}]$.
By the integral Leibnitz rule
\[
\frac{d}{dc}\int_1^{c_q/c}{\bar F}_c=
-\frac{c_q}{c^2}\, {\bar F}_c\left(\frac{c_q}{c}\right)
%\left(n(c_q)\,\frac{c_q}{c}\,\sqrt{\frac{c_q^2}{c^2}-1}\right)^{-1}
-\int_1^{c_q/c}\frac{n^\prime(tc)\, dt}{n(tc)^2\sqrt{t^2-1}}.
\]
The first summand equals 
$-(m^{\prime}(r_q)\sqrt{c_q^2-c^2})^{-1}$, and the second summand is bounded.

{\bf Step 5.}\ \ Let us investigate the behavior of
$\int_{r_q}^\infty \frac{m}{(m^2-c^2)^{3/2}} dr$ from Step 3 as $c\to c_q-$. 
Fix $\de>r_q$ such that
$m^\prime>0$ on $[r_0,\de]$ and write the above integral as the sum of the integrals
over $(r_q, \de)$ and $(\de, \infty)$. The latter one is bounded.
Integrate the former integral by parts as
\[
\int_{r_q}^\de \frac{m\, m^\prime}{m^\prime\,(m^2-c^2)^{3/2}} dr=
-\int_{r_q}^\de\frac{1}{m^\prime}d\left(\frac{1}{\sqrt{m^2-c^2}}\right)=
\]
\[
\frac{1}{m^\prime(r_q)\sqrt{c_q^2-c^2}}-
\frac{1}{m^\prime(\de)\sqrt{\de^2-c^2}}-
\int_{r_q}^\de \frac{m^{\prime\prime}\, dr}{(m^\prime)^2\sqrt{m^2-c^2}}
\]
Only the first summand is unbounded as $c\to c_q-$. 
The terms from Step 4 and 5 enter into $T^\prime$ with coefficients
$2$ and $1$, respectively, so as $c\to c_q-$ 
\[
T^\prime(c)\sqrt{c_q^2-c^2}\to -\frac{1}{m^\prime(r_q)} <0
\]
as the bounded terms multiplied by $\sqrt{c_q^2-c^2}$ disappear in the limit. 
\end{proof}

\small
\bibliographystyle{amsalpha}
\bibliography{vm}

\providecommand{\bysame}{\leavevmode\hbox to3em{\hrulefill}\thinspace}
\providecommand{\MR}{\relax\ifhmode\unskip\space\fi MR }
\providecommand{\MRhref}[2]{%
  \href{http://www.ams.org/mathscinet-getitem?mr=#1}{#2}
}
\providecommand{\href}[2]{#2}
\begin{thebibliography}{Tan92b}

\bibitem[BJ82]{BroJan}
T.~Br{\"o}cker and K.~J{\"a}nich, \emph{Introduction to differential topology},
  Cambridge University Press, Cambridge, 1982, Translated from the German by C.
  B. Thomas and M. J. Thomas.

\bibitem[Ele80]{Ele}
D.~Elerath, \emph{An improved {T}oponogov comparison theorem for nonnegatively
  curved manifolds}, J. Differential Geom. \textbf{15} (1980), no.~2, 187--216
  (1981). \MR{614366 (83b:53039)}

\bibitem[Gre97]{Gre-surv}
R.~E. Greene, \emph{A genealogy of noncompact manifolds of nonnegative
  curvature: history and logic}, Comparison geometry ({B}erkeley, {CA},
  1993--94), Math. Sci. Res. Inst. Publ., vol.~30, Cambridge Univ. Press,
  Cambridge, 1997, pp.~99--134.

\bibitem[Gro93]{Grove-crit-pt-1993}
K.~Grove, \emph{Critical point theory for distance functions}, Differential
  geometry: {R}iemannian geometry ({L}os {A}ngeles, {CA}, 1990), Proc. Sympos.
  Pure Math., vol.~54, Amer. Math. Soc., Providence, RI, 1993, pp.~357--385.

\bibitem[GS79]{GluSin}
H.~Gluck and D.~A. Singer, \emph{Scattering of geodesic fields. {II}}, Ann. of
  Math. (2) \textbf{110} (1979), no.~2, 205--225.

\bibitem[IMS03]{IMS-topon-vm}
Y.~Itokawa, Y.~Machigashira, and K.~Shiohama, \emph{Generalized {T}oponogov's
  theorem for manifolds with radial curvature bounded below}, Explorations in
  complex and {R}iemannian geometry, Contemp. Math., vol. 332, Amer. Math.
  Soc., Providence, RI, 2003, pp.~121--130.

\bibitem[KO07]{KonOht-2007}
K.~Kondo and S.-I. Ohta, \emph{Topology of complete manifolds with radial
  curvature bounded from below}, Geom. Funct. Anal. \textbf{17} (2007), no.~4,
  1237--1247.

\bibitem[KT]{KonTan-I}
K.~Kondo and M.~Tanaka, \emph{Total curvatures of model surfaces control
  topology of complete open manifolds with radial curvature bounded below.
  {I}}, arXiv:0901.4010v2 [math.DG].

\bibitem[KT10]{KonTan-II}
\bysame, \emph{Total curvatures of model surfaces control topology of complete
  open manifolds with radial curvature bounded below. {II}}, Trans. Amer. Math.
  Soc. \textbf{362} (2010), no.~12, 6293--6324.

\bibitem[KW74]{KazWar-open}
J.~L. Kazdan and F.~W. Warner, \emph{Curvature functions for open
  {$2$}-manifolds}, Ann. of Math. (2) \textbf{99} (1974), 203--219.

\bibitem[Mac10]{Mac}
Y.~Machigashira, \emph{Generalized {T}oponogov comparison theorem for manifolds
  of roughly non-negative radial curvature}, Information \textbf{13} (2010),
  no.~3B, 835--841.

\bibitem[Mae75]{Mae-2rays}
M.~Maeda, \emph{On the total curvature of noncompact {R}iemannian manifolds},
  K\=odai Math. Sem. Rep. \textbf{26} (1974/75), 95--99.

\bibitem[Men97]{Men}
S.~J. Mendon{\c{c}}a, \emph{The asymptotic behavior of the set of rays},
  Comment. Math. Helv. \textbf{72} (1997), no.~3, 331--348.

\bibitem[MS06]{MasShi}
Y.~Mashiko and K.~Shiohama, \emph{Comparison geometry referred to warped
  product models}, Tohoku Math. J. (2) \textbf{58} (2006), no.~4, 461--473.

\bibitem[Pet06]{Pet-book}
P.~Petersen, \emph{Riemannian geometry}, second ed., Graduate Texts in
  Mathematics, vol. 171, Springer, New York, 2006.

\bibitem[Sak96]{Sak-book}
T.~Sakai, \emph{Riemannian geometry}, Translations of Mathematical Monographs,
  vol. 149, American Mathematical Society, Providence, RI, 1996, Translated
  from the 1992 Japanese original by the author.

\bibitem[SST03]{SST}
K.~Shiohama, T.~Shioya, and M.~Tanaka, \emph{The geometry of total curvature on
  complete open surfaces}, Cambridge Tracts in Mathematics, vol. 159, Cambridge
  University Press, Cambridge, 2003.

\bibitem[ST02]{ShiTan-mathZ-2002}
K.~Shiohama and M.~Tanaka, \emph{Compactification and maximal diameter theorem
  for noncompact manifolds with radial curvature bounded below}, Math. Z.
  \textbf{241} (2002), no.~2, 341--351.

\bibitem[ST06]{SinTan-experim}
R.~Sinclair and M.~Tanaka, \emph{A bound on the number of endpoints of the cut
  locus}, LMS J. Comput. Math. \textbf{9} (2006), 21--39 (electronic).

\bibitem[Tan92a]{Tan-ball-poles-ex}
M.~Tanaka, \emph{On a characterization of a surface of revolution with many
  poles}, Mem. Fac. Sci. Kyushu Univ. Ser. A \textbf{46} (1992), no.~2,
  251--268.

\bibitem[Tan92b]{Tan-cut}
\bysame, \emph{On the cut loci of a von {M}angoldt's surface of revolution}, J.
  Math. Soc. Japan \textbf{44} (1992), no.~4, 631--641.

\end{thebibliography}

\end{document}